\documentclass[12pt,reqno]{amsart}
\usepackage{amsmath,amssymb,amsfonts,array,mathrsfs}
\usepackage{mathtools}
\usepackage{cases}
\usepackage{inputenc, enumerate, color}
\usepackage{graphicx}
\usepackage{fullpage}
\usepackage{times}
\usepackage{enumerate}
\usepackage{txfonts}
\usepackage{verbatim}
\usepackage[dvipsnames,usenames]{xcolor}
\usepackage[colorlinks=true,urlcolor=Black,citecolor=Black,linkcolor=Black]{hyperref}
\usepackage{url}
\usepackage{calc}
\usepackage{multirow}
\usepackage{caption}
\usepackage{relsize}

\makeatletter
\@namedef{subjclassname@2010}{%
  \textup{2010} Mathematics Subject Classification}
\makeatother

\newtheorem*{rep@theorem}{\rep@title}
\newcommand{\newreptheorem}[2]{%
\newenvironment{rep#1}[1]{%
 \def\rep@title{#2 \ref{##1}}%
 \begin{rep@theorem}}%
 {\end{rep@theorem}}}
\makeatother

\newtheorem{theorem}{Theorem}[section]
\newtheorem{remark}[theorem]{Remark}
\newtheorem{lemma}[theorem]{Lemma}
\newtheorem{corollary}[theorem]{Corollary}
\newtheorem{proposition}[theorem]{Proposition}
\newtheorem{definition}[theorem]{Definition}

\newtheorem{notation}[theorem]{Notation}
\numberwithin{equation}{section}
\numberwithin{table}{section} 
\numberwithin{figure}{section}

\theoremstyle{definition}
\newtheorem{example}[theorem]{Example}

\newcommand{\N}{\mathbb{N}}

\newcommand{\Z}{\mathbb{Z}}
\newcommand{\Q}{\mathbb{Q}}
\newcommand{\R}{\mathbb{R}}
\newcommand{\C}{\mathbb{C}}

\newcommand{\p}{\mathbf{p}}
\newcommand{\q}{\mathbf{q}}
\newcommand{\rr}{\mathbf{r}}
\newcommand{\s}{\mathbf{s}}
\newcommand{\mb }{\mathbf }

\DeclareMathOperator{\Parity}{Parity}
\DeclareMathOperator{\Sign}{Sign}

\begin{document}

\title{The Signs in Elliptic Nets}

\author{Amir Akbary, Manoj Kumar,  and Soroosh Yazdani}

\thanks{Research of the authors is partially supported by NSERC}

\date{\today}



\address{Department of Mathematics and Computer Science \\
        University of Lethbridge \\
        Lethbridge, AB T1K 3M4 \\
        Canada}
\email{amir.akbary@uleth.ca}


\email{manoj.kumar@uwaterloo.ca}
\email{syazdani@gmail.com}

\begin{abstract} 
We give a generalization of a theorem of Silverman and Stephens regarding the signs in an 
elliptic divisibility sequence to the case of an elliptic net. We also describe applications 
of this theorem in the study of the distribution of the signs in elliptic nets and 
generating elliptic nets using the denominators of the linear combination of points on 
elliptic curves.   
\end{abstract}


\maketitle

\vspace{-.75cm}
\tableofcontents
\vspace{-.75cm}

\section{Introduction}
\begin{definition}
An elliptic sequence $(W_n)$ over a field $K$ is a sequence of elements of $K$ satisfying the 
non-linear recurrence  
 \begin{equation}\label{EDS1}
			W_{m+n}W_{m-n} = W_{m+1}W_{m-1}W_{n}^2-W_{n+1}W_{n-1}W_{m}^2   
 \end{equation}  
for all $m, n\in \mathbb{Z}.$ An elliptic sequence is said to be non-degenerate if 
$W_1W_2W_3\neq 0.$  Furthermore, if $W_1=1,$ we call it a normalized elliptic sequence.  
\end{definition}
We can show that for a non-degenerate elliptic sequence $W_0=0$ (let $m=n=1$ in \eqref{EDS1}), $W_1=\pm 1$ (let $m=2$, $n=1$ in \eqref{EDS1}), and $W_{-n}=-W_n.$
The non-trivial examples of elliptic sequences can be obtained by addition of points on cubics.
Let $E$ be a cubic curve, defined over a field $K$, given by the Weierstrass equation $f(x, y)=0$, where 
\begin{equation}
    \label{WE}
    f(x,y):=y^2+a_1 xy+a_3 y- x^3-a_2 x^2-a_4 x-a_6;
    ~~~a_i\in K.
\end{equation}
Let $E^{\rm ns}(K)$ be the collection of non-singular $K$-rational points of $E.$ It is known that  $E^{\rm ns}(K)$ forms a group. Moreover, there are polynomials $\phi_n,~\psi_n$, and $\omega_n\in
\Z[a_1, a_2, a_3, a_4,a_6][x, y]$ such that
for any $P \in E^{\rm ns}(K)$ we have
$$nP=\left(\frac{\phi_n(P)}{\psi_n^2 (P)}, \frac{\omega_n(P)}{\psi_n^3(P)} \right).$$
In addition, $\psi_n$ satisfies the recursion 
\begin{equation}
    \label{dpr}
    \psi_{m+n}\psi_{m-n} =  \psi_{m+1}\psi_{m-1}\psi_{n}^{2} - \psi_{n+1}\psi_{n-1}\psi_{m}^{2}.
\end{equation}
The polynomial $\psi_n$ is called the \emph{$n$-th division polynomial} associated to $E.$ 
(See \cite[Chapter 2]{Lang1} for the basic properties of division polynomials.) The equation \eqref{dpr} shows that $(\psi_n(P))$ is an elliptic sequence over $K.$
A remarkable fact, first observed by Ward for integral (integer-valued) elliptic sequences, is that any normalized non-degenerate elliptic sequence can be realized as a sequence $(\psi_n(P)).$ A concrete version of this statement is given in the following proposition (See \cite[Theorem 4.5.3]{Swart}).

\begin{proposition}[{\bf Swart}] Let $(W_n)$ be a normalized non-degenerate elliptic sequence. Then there is a cubic curve $E$ with equation $f(x, y)=0$, where $f(x,y)$ is given by \eqref{WE} and with 
\begin{multline*}
a_1=\frac{W_4+W_2^5-2W_2W_3}{W_2^2 W_3},~a_2=\frac{W_2W_3^2+W_4+W_2^5-W_2W_3}{W_2^3 W_3},~a_3=W_2, ~a_4=1,~{\rm and}~a_6=0,
\end{multline*}
such that $W_n=\psi_n((0,0))$, where $\psi_n$ is the {$n$-th division polynomial} associated to $E.$
\end{proposition}
We call the pair $(E, (0,0))$ in the above proposition a \emph{curve-point} pair  associated with the elliptic sequence $(W_n).$ Any two curve-point pairs associated to an elliptic sequence 
$(W_n)$ are \emph{uni-homothetic} (see \cite[Section 6.2]{Stange} for definition). A normalized non-degenerate elliptic sequence $(W_n)$ is 
called \emph{non-singular} if the cubic curve $E$ in a curve-point 
$(E, P)$ associated to $(W_n)$ is an elliptic curve (a non-singular cubic).

Ward's version of the above proposition is stated for normalized, non-degenerate, integral 
elliptic divisibility sequences (i.e. an integer-valued elliptic sequence with the property that $W_m \mid W_n$ if $m\mid n$), however examining its proof reveals that in fact it is a theorem 
for any normalized, non-degenerate, elliptic sequence defined over  a subfield of $\mathbb{C}.$  
Moreover Ward represents the terms of such elliptic sequence as values of certain elliptic 
functions at certain complex numbers. 
To explain Ward's representation one observes that for the $n$-th division polynomial $\psi_n$ of an elliptic curve $E$, defined over a subfield $K
$ of $\mathbb{C}$,  we have $$\psi_n(P)=(-1)^{n^2-1} \frac{\sigma(nz;\Lambda)}{\sigma(z;
\Lambda)^{n^2}}$$
for a complex number $z$ and a lattice $\Lambda$ (See \cite[Chapter VI, Exercise 6.15]{Sil2} and \cite[Theorem 2.3.5]{Kumar} for a proof). The lattice $\Lambda$ is the 
lattice associated to $E$ over $\mathbb{C}$ and $\sigma(z;\Lambda)$ is the Weierstrass $\sigma$-
function associated to $\Lambda$ defined as
\begin{equation*}\label{sigma_function}
   \sigma(z;\Lambda) := z\prod_{\substack{\omega\in\Lambda \\ \omega\neq 0}}
  \left(1-\frac{z}{\omega}\right){e}^{\frac{z}{\omega}+\frac{1}{2}\left(\frac{z}{\omega}
  \right)^2}.
\end{equation*}
More precisely Ward proved the following assertion.

\begin{theorem}[\bf Ward]\label{thm=ward}
Let $(W_n)$ be a normalized, non-degenerate, non-singular elliptic divisibility sequence defined over a subfield $K$ of complex numbers. Then there is a 
lattice $\Lambda\subset \C$ and a complex number $z\in \C$ such that 
\begin{equation}
   W_n = \frac{\sigma(nz;\Lambda)}{\sigma(z;\Lambda)^{n^2} } \quad for~ all ~ n\geq 1.
\end{equation} 
Further, the Eisenstein series $g_2(\Lambda)$ and $g_3(\Lambda)$ associated to the 
lattice $\Lambda$ and the Weierstrass values $\wp(z;\Lambda)$ and $\wp'(z;\Lambda)$ associated 
to the point $z$ on the elliptic curve $\C/\Lambda$ are in the field $\Q(W_2,W_3,W_4).$ In 
other words $g_2(\Lambda), g_3(\Lambda), \wp(z;\Lambda), \wp'(z;\Lambda)$ are all defined over 
$K.$
 \end{theorem}

The above version of Ward's theorem is \cite[Theorem 3]{Sil1}.
In \cite{Sil1} Silverman and Stephens proved a formula regarding signs in an unbounded,  normalized, 
non-degenerate, non-singular, real elliptic sequence. (The results of  \cite{Sil1} stated for integral elliptic divisibility sequences, however their results hold more generally for real elliptic sequences.) In order to describe Silverman-Stephens's theorem we need to set up 
some notation.
\begin{notation}
\label{Notation}
For an elliptic curve $E$ defined over $\R,$ we let $\Lambda \subset\C $ be its corresponding
lattice. Let $E(\R)$ be the group of $\R$-rational points of $E.$ For a point $P\in E(\R)$
we let $z$ be the corresponding complex number under the isomorphism $E(\C)\cong\C/\Lambda.$ 
From the theory of elliptic curves we know that there exists a 
unique $q=e^{2\pi i \tau },$ where $\tau$ is in the upper half-plane, such that 
$\R^*/q^{\Z}\cong E(\R)$ (see Theorem \ref{R_Isomorphism}). Let $u\in \R^*$ be the corresponding real number to the point $P\in 
E(\R)$, where $P\neq \mathcal{O}$ (the point at infinity).We assume that $u$ is normalized such that it satisfies $q<|u|<1$ if $q>0$ and $q^2<u<1$ if $q<0$ (see Lemma \ref{normalize_u}).
Finally, for any non-zero real number $x,$ we define the parity of $x$ by

\begin{equation*}
     \Sign[x] = (-1)^{\Parity[x]}  \qquad \text{where }\Parity[x] \in \Z/2\Z.  
     \end{equation*}
     
\end{notation}

The following is \cite[Theorem 1]{Sil1}.

\begin{theorem}[\bf Silverman-Stephens]\label{thm=Sil-Ste}
Let $(W_n)$ be an unbounded, normalized, non-singular, non-degenerate (integral) elliptic divisibility sequence. Let $(E,P)$ be
a curve-point corresponding to $(W_n).$ Assume conventions given in Notation \ref{Notation}. Then possibly after replacing $(W_n)$ by the related sequence $((-1)^{n^2-1}W_n),$ there is an irrational number $\beta \in \R$, given in Table \ref{table:silbeta}, so that if $q<0$, or $q>0$ and $u>0$, 
\begin{equation*}
\Parity[W_n] \equiv {\lfloor n\beta \rfloor} \pmod 2 ,
\label{u>0} 
\end{equation*}
and if $q>0$ and $u<0$,
\begin{equation*}
\Parity[W_n] \equiv
           \begin{cases}  {\lfloor n\beta \rfloor+n/2} \pmod 2& ~if~ n~ is~ even, \\
                          {(n-1)/2}\hskip.2cm      \pmod 2             &  ~if~ n~ is~ odd.
            \end{cases}\label{u<0}
\end{equation*}
Here $\lfloor . \rfloor$ denotes the greatest integer function.

\begin{table}[h]
\begin{center}
\begin{tabular}{| >{\centering\arraybackslash}m{0.9in} 
                | >{\centering\arraybackslash}m{1in} 
                | >{\centering\arraybackslash}m{1.5in} 
                |}
\hline
\hline
$q$ & $u$ & $\beta$   \parbox{0pt}{\rule{0pt}{2ex+\baselineskip}}\\ \hline
\multirow{3}{*}{$q>0$} 
   & $u>0$ &  $\log_{q}u $ 
                                        \parbox{0pt} {\rule{0pt}{1.5ex+\baselineskip}} \\
   & $u<0$ & $\log_{q}|u| $ 
                                        \parbox{0pt}{\rule{0pt}{1.5ex+\baselineskip}} \\ \hline
$q<0$  & $u>0$ & $\displaystyle \frac{1}{2}\log_{|q|}u $
                                        \parbox{0pt}{\rule{0pt}{4ex+\baselineskip}}\\ 
\hline
\hline
\end{tabular}\vspace{0.35cm}
\caption{Explicit expressions for $\beta$ }
\label{table:silbeta}
\end{center}
\end{table}
\end{theorem}
In this paper we give a generalizations of Silverman-Stephens's theorem in the context of elliptic 
nets. 
\begin{definition}\label{NP}
Let $\Lambda \subset \C $ be a fixed lattice corresponding to an elliptic curve $E/\C.$ For an 
$n$-tuple $\mathbf{v}=(v_1,v_2,...,v_n) \in \Z^n,$ define a function $\Omega_{\mathbf{v}}$ (with 
respect to $\Lambda$) on $\C^n$ in variable $\mathbf{z}=(z_1,z_2,...,z_n)$ as follows:
\begin{equation}\label{rank_n_net} 
\Omega_{\mathbf{v}}(\mathbf{z};\Lambda)= (-1)^{\displaystyle {\sum_{i=1}^{n}v_i^2- \hspace{-0.7em} \sum_{1\leq i <j \leq n}\hspace{-0.7em}v_iv_j-1}}\hspace{-0.5em}\frac{\sigma(v_1z_1+v_2z_2+...+v_nz_n;\Lambda)}{\displaystyle \prod_{i=1}^{n}\sigma(z_i ;\Lambda)^{2v_i^2-\sum_{j=1}^n v_iv_j } \prod_{1\leq i <j \leq n}\sigma(z_i+z_j ;\Lambda)^{v_iv_j} },
\end{equation}
where $\sigma(z;\Lambda)$ is the Weierstrass $\sigma$-function.
\end{definition}
In the above definition, the complex points $z_{i}$ satisfy $P_{i} = (\wp(z_{i}), \wp'(z_{i})/2)$ for each 
$1 \leq i\leq n$ (under some embedding of $K \rightarrow \C$), where $\wp(z)$ is the Weierstrass 
$\wp$-function attached to $\Lambda$. Let ${\bf P}=(P_1, P_2, \ldots, P_n)$. 
In \cite[Theorem 3.7]{Stange} it is shown that if
$\mb{P} = (P_1,P_2, \ldots, P_n) $ is an $n$-tuple consisting of $n$ points in $E(\C)$ such that 
$P_i\neq \mathcal{O} $ for each $i$ and $P_i\pm P_j\neq \mathcal{O} $ for $1\leq i<j \leq n$, and 
$\mathbf{z}=(z_1,z_2,...,z_n)$ in $\C^n $ be such that each $z_i$ corresponds to $P_i$ under the 
isomorphism $\C/\Lambda \cong E(\C),$ then 
$\Omega_{\mathbf{v}}:=\Omega_{\mathbf{v}}(z;\Lambda)$ satisfies
the recursion 
\begin{equation} \label{npr}
\Omega_{\mathbf{p}+\mathbf{q}+\mathbf{s}}\Omega_{\mathbf{p}-\mathbf{q}}\Omega_{\mathbf{r+s}}
\Omega_{\mathbf{r}}+\Omega_{\mathbf{q+r+s}}\Omega_{\mathbf{q-r}}\Omega_{\mathbf{p+s}}
\Omega_{\mathbf{p}}+\Omega_{\mathbf{r+p+s}}\Omega_{\mathbf{r-p}}\Omega_{\mathbf{q+s}}
\Omega_{\mathbf{q}} = 0,
\end{equation} 
for all $\p, \q, \rr, \s \in \Z^n.$ 
In \cite{Stange}, Stange generalized the concept of an elliptic sequence to an
 $n$-dimensional array, called an elliptic net. 

 \begin{definition} 
	 Let $A$ be
	 a free Abelian group of finite rank, and $R$ be an integral domain. Let
	 ${\mathbf 0}$ and $0$ be the additive identity elements of $A$ and $R$
	 respectively. An {elliptic net} is any map $W:A \rightarrow R$ for which
	 $W({\mathbf 0}) = 0$, and that satisfies 
	 \begin{multline} \label{net recurrence} 
		 W(\p+\q+\s)W(\p-\q)W(\rr+\s)W(\rr) \\ +
		 W(\q+\rr+\s)W(\q-\rr)W(\p+\s)W(\p) \\ + 
		 W(\rr+\p+\s)W(\rr-\p)W(\q+\s)W(\q) = 0, 
	 \end{multline} 
	 for all $\p,\q,\rr,\s \in A.$ We identify the {rank}
	 of $W$ with the rank of $A.$  
 \end{definition}
 Note that for $A=\mathbb{Z}$, $\s=0$, $\rr=1$, and $W(1)=1$ the recursion \eqref{net recurrence} reduces to \eqref{EDS1}. Thus elliptic nets are generalizations of elliptic sequences. Moreover, in light of \eqref{npr} the function 
\begin{align*}
\Psi(\mb{P};E): \Z^n \quad &\longrightarrow \quad \C \\
                \mb{v}\quad &\longmapsto ~\Psi_{\mathbf{v}}(\mb{P};E)=\Omega_{\mb{v}}(\mb{z};\Lambda)
\end{align*}
is an elliptic net with values in $\C.$ 
Observe that $\Psi_{n{\bf e_i}} ({\bf P})=\psi_n(P_i)$, where $\mathbf{e_{i}}$ 
denotes the $i^{th}$ standard basis vector for $\Z^{n}.$


\begin{definition}\label{PsiPE}
The function $\Psi(\mb{P};E)$ is called the {elliptic net associated 
to $E$ (over $\C$) and $\mb{P}.$} The value $\Psi_{\mathbf{v}}(\mb{P};E)=\Omega_{\mb{v}} (\mb{z};\Lambda)$ is called the {$\mathbf{v}$-th net polynomial} 
associated to $E$ and $\mb{P}$. \end{definition}
 
We note that
if $P_1, P_2, \ldots, P_n$ are $n$ linearly independent points in $E(\mathbb{R})$ then by \cite[Theorem 7.4]{Stange} we have $\Psi_{\mathbf{v}}(\mb{P};E)\neq 0$ for $\mb{v}\neq \mb{0}$.
We prove the following generalization of Theorem \ref{thm=Sil-Ste} regarding the signs in
$\Psi(\mb{P},E)$.
\begin{theorem}\label{main_theorem_1}
Let $E$ be an elliptic curve defined over $\R$ and $\mathbf{P} = (P_1,P_2,\ldots, P_n)$ be an 
$n$-tuple consisting of $n$ linearly independent points in $E(\R).$ Let $\Lambda, q, z_i,$ 
and $u_i$ be as defined in Notation \ref{Notation}. Assume that $u_1,u_2, \ldots,u_n >0$ or there 
exists a non-negative integer $k$ such that $ u_1, u_2,\ldots, u_k <0 $ and $u_{k+1},u_{k+2},
\ldots, u_n >0.$ 
Then there are $n$ irrational numbers $\beta_1, \beta_2,\ldots,
\beta_n, $ which are $\Q$-linearly independent, given by rules similar to Table \ref{table:silbeta}, 
such that the parity of $~\Psi_{\mb{v}}(\mb{P};E)\left(=\Omega_{\mb{v}}(\mb{z};\Lambda )\right)$, 
possibly after replacing $~\Psi_{\mb{v}}(\mb{P};E)$ with 
{\tiny $(-1)^{ \displaystyle {\sum_{i=1}^{n} v_i^2-\hspace{-0.7em} \sum_{1\leq i <j \leq n} 
\hspace{-0.7em}v_iv_j-1}}$}\hspace{-0.7em} $\Psi_{\mb{v}}(\mb{P};E),$ is given by 
\begin{align} 
\Parity[\Psi_{\mb{v}}(\mb{P};E)] &\equiv
              \left\lfloor \sum_{i=1}^{n}v_i\beta_i \right\rfloor + \sum_{1\leq i <j\leq n} 
              \lfloor \beta_i+\beta_j  \rfloor v_iv_j \pmod 2 ,  
\label{eq=for_k=0} 
\end{align}
if all $u_i >0$, and

\begin{subequations}
    \begin{numcases} 
    {\Parity [\Psi_{\mb{v}}(\mb{P};E)] \equiv}
    \displaystyle \sum_{1\leq i<j\leq k}\lfloor \beta_i+\beta_j \rfloor v_iv_j \nonumber
                    + \sum_{k+1\leq i<j\leq n }\lfloor \beta_i+\beta_j \rfloor v_iv_j \\
      +\displaystyle \Big\lfloor \sum_{i=1}^{n}v_i\beta_i \Big\rfloor 
                              + \sum_{i=1}^{k} \Big\lfloor \frac{v_i}{2} \Big\rfloor 
         \pmod 2 & \hspace{-1cm} \text{if $ \displaystyle ~\sum_{i=1}^{k}v_i$ is even,}\label{eq=for_k>0a} \\  \nonumber \\ 
    \displaystyle \sum_{1\leq i<j\leq k} \lfloor \beta_i+\beta_j \rfloor v_iv_j \nonumber
                + \sum_{k+1\leq i<j\leq n} \lfloor \beta_i+\beta_j \rfloor v_iv_j\\ 
      +\displaystyle  \sum_{i=1}^{k} \Big\lfloor \frac{v_i}{2}\Big \rfloor 
           \pmod 2 &\hspace{-0.9cm} \text{if $\displaystyle ~\sum_{i=1}^{k}v_i$ is odd, }\label{eq=for_k>0b}
    \end{numcases}
\end{subequations}
if $ u_1, u_2,\ldots, u_k <0 $ and $u_{k+1},u_{k+2}, \ldots, u_n >0.$ 
\end{theorem}
Note that in the above theorem all $u_i>0$ is the same as $k=0$, which leads to $\sum_{i=1}^{k} v_i=0$ always being even. Thus
\eqref{eq=for_k>0a} for $k=0$ reduces to \eqref{eq=for_k=0}. The method of the proof of the above theorem follows closely the techniques devised in the proof of Theorem 1 of \cite{Sil1} for the case $n=1$, however the proof of Theorem \ref{main_theorem_1}
 involves analyzing more cases since the expression \eqref{rank_n_net}, for $n>1$, includes some new terms.

We also prove a generalization of Theorem \ref{thm=Sil-Ste} for sign of certain elliptic nets that are not necessarily given as values of net polynomials. In order to describe our result, we need to review some concepts from the theory of elliptic nets as developed in \cite{Stange}.

\begin{definition}
Let $W: \Z^n \longrightarrow R$ be an elliptic net. Let $\mathcal{B} = \{\mathbf{e_1,e_2,\ldots,e_n}
\}$ be the standard basis of $\Z^n.$ We say that $W$ is {non-degenerate} if  
$W(\mb{e_i}),~W(\mb{2e_i}) \neq 0 $ for all $1\leq i\leq n,$ and  $ W(\mb{e_i\pm e_j})\neq 0 $ for 
$1\leq i,j \leq n,~i\neq j.$ If $n=1$, we need an additional condition that 
$W(\mb{3e_i})\neq 0.$ If any of the above conditions is not satisfied we say that $W$ is 
{degenerate}. 
\end{definition}

\begin{definition}
Let $W: \Z^n \longrightarrow R$ be an elliptic net. Then we say that $W$ is {normalized} if 
$W(\mb{e_i}) =1$ for all $1\leq i\leq n$ and $W(\mb{e_i+e_j}) =1$ for all $1\leq i<j \leq n.$
\end{definition}
In \cite[Theorem 7.4]{Stange} a generalization of Theorem \ref{thm=ward} in the context of elliptic nets is given. More precisely it is proved that for a normalized and non-degenerate elliptic net $W: \Z^n \longrightarrow K $ there exists a cubic curve $E$ and a collection of points $\mathbf{P}$ on $E$ such that $W$ can be realized as an elliptic net associated to $E$ and $\mathbf{P}.$ (Theorem 7.4 of \cite{Stange} is also applicable to elliptic nets over a field $K$ that is not contained in $\mathbb{C}$.) We call $W$ \emph{non-singular} if $C$ in the curve point $(E, \mathbf{P})$ associated to $W$ is an elliptic curve. We also need the following concept for our second generalization of Theorem \ref{thm=Sil-Ste}.

\begin{definition}\label{QF}
A function 
$f:\mathbb{Z}^n\longrightarrow \mathbb{R}^*$ is called a {quadratic form} if 

\begin{equation}\label{parallelogram_law2}
f(\mb{a+b+c})f(\mb{a+b})^{-1}f(\mb{b+c})^{-1}f(\mb{c+a})^{-1}f(\mb{a})f(\mb{b})f(\mb{c}) = 1,
\end{equation}
for $\mb{a, b, c} \in \mathbb{Z}^n$.
 \end{definition}
 An example of a quadratic form is the function
    \begin{equation*}
 f(v_1,v_2,\ldots,v_n) = \prod_{i=1}^{n} p_i^{v_i^2}\prod_{1\leq i<j \leq  n}q_{ij}^{v_iv_j},
\end{equation*}
where $p_i, q_{ij}\in \mathbb{R}^*.$
 As we mentioned before, Theorem \ref{main_theorem_1} can be stated as a theorem for the sign of certain elliptic nets. Our next theorem establishes such a result for non-singular, non-degenerate elliptic nets. 

\begin{theorem}\label{main_theorem_2}
Let $W : \Z^n \longrightarrow \R $ be a non-singular, non-degenerate elliptic net. Assume that $W(\mb{v})\neq 0$ for $\mb{v}\neq \mb{0}$. Then, possibly 
after replacing $W(\mb{v})$ with either $g(\mb{v}) W(\mb{v})$  or $-g(\mb{v}) W(\mb{v})$ for a quadratic form $g : \Z^n 
\longrightarrow \R^*$, there are $n$ irrational numbers $\beta_1, \beta_2,\ldots ,\beta_n$, given by 
rules similar  to Table \ref{table:silbeta}, that can be calculated using an elliptic curve associated to $W$ and points 
on it, such that  
%
\begin{align}
\Parity[W(\mb{v})] &\equiv
  \left\lfloor \sum_{i=1}^{n}v_i\beta_i \right\rfloor \pmod 2 .  \label{eq=sign1} \\
\Parity[W(\mb{v})]&\equiv
\begin{cases}
    \displaystyle \Big\lfloor \sum_{i=1}^{n}v_i\beta_i \Big\rfloor + \sum_{i=1}^{k} \Big\lfloor \frac{v_i}{2} \Big\rfloor 
         \pmod 2 & \text{if $ \displaystyle ~\sum_{i=1}^{k}v_i$ is even } \\ 
\displaystyle ~ \sum_{i=1}^{k} \Big\lfloor \frac{v_i}{2}\Big \rfloor \pmod 2 & \text{if $\displaystyle ~\sum_{i=1}^{k}v_i$ is odd }, \label{eq=sign2}
\end{cases}
\end{align} 
where \eqref{eq=sign1} is applicable when all $u_i>0$ and \eqref{eq=sign2} is applicable if $u_1, u_2, \ldots, u_k <0$ and $u_{k+1}, u_{k+2}, \ldots, u_n >0$ .
\end{theorem}

Again note that for $k=0$  the formula \eqref{eq=sign2}
reduces to  \eqref{eq=sign1}.  Next we describe some applications of Theorem \ref{main_theorem_1} and Theorem \ref{main_theorem_2}.
\begin{definition}\label{def=ud_grank_modm}
For $\mathbf{v} = (v_1, v_2,  \ldots,  v_n)\in \N^n,$ let $(S(\mb{v}))$ be an $n$-dimensional array of integers. For $j\in \{0, 1, \ldots, m-1\}$ and  $m\geq 2$ denote
\begin{equation*}
C(m, j;V_1,V_2,\ldots,V_n) = 
\# \big\{  \mb{v}; ~1  \leq  v_i \leq  V_i \hspace{.5em} \mbox{for}~ 1 \leq i \leq n~and
                 ~S(\mb{v}) \equiv j \pmod m  \big\}.
\end{equation*}
The array $(S(\mb{v}))$ is said to be uniformly distributed mod $m$ if 
\begin{equation*}
\lim_{V_1,V_2,\ldots , V_n \rightarrow \infty} \frac{C(m, j;V_1,V_2,\ldots,V_n)}{V_1V_2\ldots V_n}
                                        =\frac{1}{m},
\end{equation*}
for $j = 0, 1, \ldots ,m-1.$
We say that the signs in an  $n$-dimensional array $S: \mathbb{Z}^n \rightarrow \mathbb{R}^*$ is {uniformly distributed} if the array $(\Parity [S(\mathbf{v})])$ is uniformly distributed mod $2.$
\end{definition}
Note that here the restriction to $\mathbf{v}\in \mathbb{N}^n$   is for the simplicity of presentation and similar results will hold for $\mathbf{v}\in \mathbb{Z}^n$. By employing formulas in Theorem \ref{main_theorem_1} and Theorem \ref{main_theorem_2} we establish the following result.
\begin{theorem}
\label{ud-sign}
Let $ \Psi(\mb{P};E)$ and $W(\mb{v})$ be as in Theorem \ref{main_theorem_1} and Theorem \ref{main_theorem_2}. Then the signs in $ \Psi(\mb{P};E)$ and $W(\mb{v})$ are uniformly distributed.

\end{theorem}

In order to explain the second application of our results we first introduce the concept of a \emph{denominator 
net}.
Let $E/\Q$ be an elliptic curve given by a Weierstrass equation with integer coefficients. If $P \in 
E(\Q)$ is a non-torsion point (i.e., $nP \neq \mathcal{O}$ for any $n$) then we have that 
\begin{align*}
nP &= \left( \frac{A_{nP}}{D_{nP}^2}, ~\frac{B_{nP}}{D_{nP}^3} \right), 
\end{align*}
where $A_{nP}, B_{nP},$ and $D_{nP}>0$ are integers
(See \cite[Chapter III, Section 2]{Sil0}).
The sequence $(D_{nP})$ is called an \emph{elliptic denominator sequence} associated to the curve $E$ and 
the point $P.$ It can be shown that $(D_{nP})$ is a divisibility sequence. Several authors have 
studied the sequence $(D_{nP}).$ In fact, Shipsey \cite{Shipsey} has shown a way of assigning signs to the sequence $(D_{nP})$ so that the resulting sequence becomes an 
elliptic divisibility sequence (Note that $D_{nP}>0$ for all $n$ by our definition). More precisely, let $E$ be an elliptic curve given by
\begin{equation}\label{specialE}
y^2 + a_1xy + a_3y = x^3+a_2x^2 + a_4x, \qquad a_i \in \Z 
\end{equation}
with the condition that $\gcd(a_3,a_4)=1.$ Let $P=(0,0)$ be a point of infinite order
in $E(\Q).$ Let $(D_{nP})$ be the associated elliptic denominator sequence. Let $(W_n)$ be an array
defined by the rule given as
\begin{equation*}
W_1 = 1, \qquad W_2 = a_3, \qquad |W_n| = D_{nP} ~~\mbox{for } n\geq 2.
\end{equation*}
Suppose we assign signs to the terms of $(W_n)$ by the rule
\begin{equation*}
\Sign[W_{n-2}W_n] = -\Sign[A_{(n-1)P}] ~~\mbox{for } n\geq 3.
\end{equation*}
Then in \cite[Section 4.4]{Shipsey} it is shown that $(W_n)$ will be an elliptic divisibility 
sequence. We observe that the condition on $E$ that $\gcd(a_3,a_4)=1$ is equivalent to that $P = 
(0,0)$ reduces to a non-singular point modulo any prime $\ell.$  It is also shown that if a curve is not of the form \eqref{specialE} 
then it is always possible to transform it into a curve of the form \eqref{specialE} (see
\cite[Chapter 5]{Shipsey}).

The concept of an elliptic denominator sequence has been generalized to higher ranks and it is called an 
\emph{elliptic denominator net}. If $\mathbf{P} =  (P_1,P_2,\ldots,P_n)$ is an $n$-tuple of 
linearly independent points in $E(\Q).$ Then for $\mathbf{v} = (v_1,v_2,\ldots,v_n )\in \Z^n $ we 
can write 
\begin{equation*}
\mathbf{v\cdot P} = v_1P_1+v_2P_2+\cdots+v_nP_n = 
\left( \frac{A_{\mb{v}\cdot P}}{D^2_{\mb{v}\cdot P}},\frac{B_{\mb{v}\cdot P}}{D^3_{\mb{v}\cdot P}} \right). 
\end{equation*}
Then $(D_{\mb{v}\cdot P})$ is called the \emph{elliptic denominator net} associated to an elliptic 
curve $E$ and a collection of points $\mathbf{P}.$  
As a consequence of Theorem \ref{main_theorem_1} and \cite[Proposition 1.7]{ABY}, our final result
describes how one can assign signs to a denominator net in order to obtain an elliptic net.
\begin{theorem}
\label{one-seven}
Let $E$ be an elliptic curve defined over $\Q$ given by the Weierstrass equation
\begin{equation*}
y^2 + a_1xy + a_3y = x^3+a_2x^2 + a_4x + a_6, \qquad a_i \in \Z. 
\end{equation*}
Let $\mathbf{P} =  (P_1,P_2,\ldots,P_n)$ be an $n$-tuple of linearly independent points in $E(\Q)$
so that each $P_i \pmod{\ell}$ is non-singular for every prime $\ell .$
Define a map $W :\Z^n \longrightarrow \Q $  as 
\begin{equation}\label{DvP}
W(\mb{v}) = (-1)^{\Parity[\Psi_{\mb{v}}(\mb{P};E)]} D_{\mb{v}\cdot P}, 
\end{equation}
where $\Psi(\mb{P};E)$ is the elliptic net associated to $E$ and the collection of points 
$\mb{P}.$ Then $W$ is an elliptic net.
\end{theorem}
Note that for rank one elliptic nets the above theorem gives an alternative method, different from Shipsey's, 
in generating elliptic sequences out of elliptic denominator sequences.  

In the next section we will  review preliminaries needed in the proofs and in Sections 3 and 4 we  prove our main results on the signs in elliptic nets. In Section 5 we illustrate our results by providing several examples. Finally in Sections 6 and 7 we give proofs of our results on uniform distribution of signs and on relation with denominator sequences. 


\section{Preliminaries}

We will follow the conventions described in Notation \ref{Notation}. We first show that the claimed normalization in Notation \ref{Notation} is possible. 
\begin{lemma}\label{normalize_u}
Let $ q\in \R$ be such that $0<|q|<1$ and $ u_0\in \R^{>0}\setminus q^{\mathbb{Z}}$. Then we have the following statements.
\begin{enumerate}[(i)]
\item For $0<q<1$ there exists an integer $k$ such that $0<q<q^ku_0<1.$
\item For $-1<q<0$ there exists an integer $k$ such that $0<q^2<q^ku_0<1.$
\end{enumerate}
\end{lemma}
\begin{proof}
(i)  Let $k_0= \min \{k \in \Z ~|~q^ku_0 <1\}.$ Then $q^{k_0}u_0<1$ and $q^{k_0-1}u_0>1.$ We claim that
      $q<q^{k_0}u_0<1.$ Clearly $q^{k_0}u_0<1.$ If $q^{k_0}u_0\leq q$ then $q^{k_0-1}u_0\leq 1$ which 
      contradicts the minimality of $k_0.$ So the claim holds.
      
(ii) If $-1<q<0,$ then $0<q^2<1,$ so the result follows from part (i). 
\end{proof} 
\noindent Thus, letting $u=q^ku_0$ in the above lemma will result in the desired normalization.

Let 
$\Lambda_{\tau} $ be the normalized lattice with basis $[\tau, 1],$ where $\tau 
$ is in the upper half-plane. From \cite[Chapter I, Theorem 6.4]{Sil3} we know that, the $q$-expansion of the $\sigma$-function 
$\sigma(z;\Lambda_{\tau} ) $ is given by 
\begin{equation}\label{sigma_q_expansion_tau}
\sigma(z;\Lambda_{\tau} )=-\frac{1}{2\pi i} e^{\frac{1}{2} z^{2}\eta-\pi iz}(1-w)\prod_{m\geq 1} 
\frac{(1-q^mw)(1-q^mw^{-1})}{(1-q^m)^2},
\end{equation}
where $w = e^{2\pi i z}, ~q = e^{2\pi i \tau},$ and $\eta $ is the quasi-period homomorphism. 
The next proposition gives the $q$-expansion for the numerator in the expression for 
$\Omega_{\mathbf{v}}(\mathbf{z};\Lambda_{\tau} )$ in \eqref{rank_n_net}.
  
\begin{proposition}\label{ranknsigma}
Let $\mb{v}=(v_1,v_2,\dots , v_n) \in \Z^n $ and $\mb{z}=(z_1,z_2, \dots , z_n)\in \C^n.$
Let $w_j = e^{2\pi iz_j}$ for $j= 1,2,\dots n$ and $q=e^{2\pi i\tau}.$ Then 
\begin{equation}\label{rank_n_sigma} 
\sigma(\mb{v}\cdot \mb{z}; \Lambda_{\tau})
      =-\frac{1}{2\pi i} e^{\frac{1}{2}(\mb{v}\cdot \mb{z})^{2}\eta 
          -\pi i(\mb{v}\cdot \mb{z})}\Big(1- \prod_{j=1}^{n} w_j^{v_j}\Big)
          \prod_{m\geq 1}\frac{(1-q^m\prod_{j=1}^{n} w_j^{v_j})(1-q^m\prod_{j=1}^{n} 
          w_j^{-v_j})}{(1-q^m)^2},
\end{equation}
where $\mb{v}\cdot \mb{z} = v_1z_1+v_2z_2+...+v_nz_n.$
\end{proposition}
\begin{proof}
The result is obtained by replacing $z$ with $\mb{v}\cdot \mb{z} =  
v_1z_1+v_2z_2+ +\cdots v_nz_n $ in \eqref{sigma_q_expansion_tau}. Observe that the map 
$z \longmapsto  v_1z_1+v_2z_2+ +\cdots v_nz_n,$ corresponds to $w \longmapsto \prod_{j=1}^{n}
w_j^{v_j} .$
\end{proof}

The next proposition provides a $q$-expansion for $\Omega_{\mb{v}}(\mb{z};\Lambda_{\tau})$ defined in Definition \ref{NP}.
\begin{proposition}\label{Omega_in_theta}
Let $\mb{v}=(v_1,v_2,\dots , v_n) \in \Z^n $ and $\mb{z}=(z_1,z_2, \dots , z_n)\in \C^n.$
Let $w_j = e^{2\pi iz_j}$ for $j= 1,2,\dots n$ and $q=e^{2\pi i\tau}.$  
Then we have
\begin{equation*}
\Omega_{\mb{v}}(\mb{z};\Lambda_{\tau})= ({2\pi i})^{\displaystyle{ \sum_{j=1}^{n}v_j^2-
\hspace{-0.7em} \sum_{1\leq j <k \leq n}\hspace{-0.7em}v_jv_k-1}}\prod_{j=1}^{n} w_j^{
 \frac{v_j^2-v_j}{2}}\frac{\displaystyle \mathlarger{\theta}\Big(\prod_{j=1}^{n} w_j^{v_j},q \Big)}
{\displaystyle \prod_{j=1}^{n} \theta(w_j, q)^{2v_j^2-\sum_{k=1}^n v_jv_k }\hspace{-0.7em}
\prod_{1\leq j<k \leq n}\hspace{-0.7em}\theta(w_jw_k, q)^{v_jv_k} },
\end{equation*}
 where 
\begin{align*}
\theta(w_j,q) &= (1-w_j)\prod_{m\geq 1} \frac{(1-q^mw_j)(1-q^mw_j^{-1})}{(1-q^m)^2}.
\end{align*}
\end{proposition}
\begin{proof}
The proof is computational and follows by substituting the $q$-expansions 
\eqref{sigma_q_expansion_tau} and \eqref{rank_n_sigma} in \eqref{rank_n_net}. The one thing to 
note is that the product expansion of $\Omega_{\mb{v}}(\mb{z};\Lambda_{\tau})$ is independent of 
$\eta,$ the quasi-period homomorphism. It disappears after substituting the $q$-expansions and simplifying the terms. 
\end{proof}

For $q = e^{2 \pi i \tau}$ with $\tau$ in the upper half-plane, let $E_q$ be the elliptic curve defined as 
\begin{equation*}
E_q : y^2 + xy = x^3+a_4(q)x + a_6(q), 
\end{equation*}\vspace{-1em}
where 

\begin{align*}
a_4(q) &= -5\sum_{n\geq 1}\frac{n^3q^n}{1-q^n} \\
\end{align*}
and
\begin{align*}
a_6(q) &= -\frac{5}{12}\sum_{n\geq 1}\frac{n^3q^n}{1-q^n}-\frac{7}{12}\sum_{n\geq 1}
                                                                \frac{n^5q^n}{1-q^n}.
\end{align*}
Let 
\begin{align}\label{Canaiso}
\phi : \C^*/q^{\Z} &\overset{\sim} \longrightarrow E_q(\C) 
\end{align}
be the $\C$-analytic isomorphism given in \cite[Chapter V, Theorem 1.1]{Sil3}.
We are only concerned with elliptic nets $\Psi(\mb{P};E)$ with values in $\R.$ By 
\cite[Theorem 4.4]{Stange} if $E$ is defined over $\R,$ then we have $\Psi_{\mb{v}}(\mb{P};E)\in \R$ 
for any $\mb{v} \in \Z^n .$ So from now on we assume that our elliptic curves are defined over 
$\R.$
The following theorem will play an important role in our investigations.
\begin{theorem}\label{R_Isomorphism}
Let $E/\R$ be an elliptic curve. Then the following assertions hold. 
\begin{enumerate}[(a)]
\item There is a unique $q\in \R$ with $0<|q|<1$ such that $$E \cong_{/\R } E_q $$
    (i.e., $E$ is $\R$-isomorphic to $E_q$).  
\item The composition of the isomorphism in part (a) with the isomorphism $\phi$ defined in 
\eqref{Canaiso}, yields an isomorphism  
      \begin{equation*}
      \psi : \C^*/q^{\Z} \overset{\sim} \longrightarrow E(\C)
      \end{equation*}
      which commutes with complex conjugation. Thus $\psi$ is defined over $\R$ and      
      moreover,
      \begin{equation*}
      \psi : \R^*/q^{\Z} \overset{\sim} \longrightarrow E(\R)
      \end{equation*}
      is an $\R$-analytic isomorphism. 
\end{enumerate}
\end{theorem}
\begin{proof}
See \cite[Chapter V, Theorem 2.3]{Sil3}.
\end{proof}

Since $E$ is defined over $\R,$ by Theorem \ref{R_Isomorphism}, there is a unique real number 
$q\in \R$ with $0<|q|<1$ such that $E \cong_{/\R } E_q.$ Assume that $\pi: E_q \rightarrow E$ represents this 
isomorphism.
Let $\tau$ be  a complex number associated to $q$ given in 
Theorem \ref{R_Isomorphism} such that $q = e^{2 \pi i \tau}$ and let $\Lambda_{\tau}$ be the lattice 
generated by $[\tau, 1].$ Since $E \cong E_q ,$ there exists an $\alpha \in \C^*$ such that $\Lambda 
= \alpha\Lambda_{\tau}$, where $\Lambda$ is the lattice associated with $E$.  The multiplication by $\alpha$ carries $\C/\Lambda $ isomorphically to 
$\C/\Lambda_{\tau}.$ Let $z_i$ be 
the corresponding complex number to $P_i\in E(\R)$ under the isomorphism $E(\C)\cong C/\Lambda.$ Then 
$z_i/\alpha $ is the corresponding complex number to $\pi^{-1}(P_i)\in E_q(\R)$ under the isomorphism 
$E_q(\C)\cong C/\Lambda_{\tau}.$ From part (b) of Theorem \ref{R_Isomorphism}, the map
\begin{equation*}
\psi = \pi  ~o ~\phi : \C^*/q^{\Z} \overset{\sim} \longrightarrow E_{q}(\C) \overset{\sim} \longrightarrow E(\C)
\end{equation*}
is an isomorphism, moreover the map $\psi$ (restricted to $\R^*/q^{\Z}$)
\begin{equation*}
\psi : \R^*/q^{\Z} \overset{\sim} \longrightarrow E_{q}(\R) \overset{\sim} \longrightarrow E(\R)
\end{equation*} is an $\R$-isomorphism. Thus from construction of $\psi,$ we can consider 
$u_i = e^{2 \pi i z_i/\alpha}$ as a representative in $\R^*/q^{\Z}$ for $\psi^{-1}(P_i).$
Since $\psi$ is an $\R$-isomorphism we have that $u_i\in \R^*. $

Next let  $ \Psi_{\mb{v}} (\mb{P},E) = \Omega_{\mb{v}}(\mb{z};\Lambda ) $ be the value of the 
$\mb{v}$-th net polynomial at $\mb{P}.$ Then for $ \mb{v} =(v_1,v_2,\dots,v_n)\in \Z^n,$ fixed 
$\mb{z}= (z_1,z_2,\dots,z_n)\in \C^n,$ and $\Lambda,$ we have 
\begin{equation*}\label{Omega_Omega}
\Omega_{\mb{v}}(\mb{z};\Lambda) = \Omega_{\mb{v}}(\mb{z};\alpha\Lambda_{\tau})
        = (\alpha^{-1})^{\displaystyle \sum_{i=1}^{n} v_i^2-\hspace{-0.7em}\sum_{1\leq i <j \leq n}\hspace{-0.7em} v_iv_j-1 }\Omega_{\mb{v}}(\mb{\alpha^{-1}z};\Lambda_{\tau}).
\end{equation*}
Here we have used the fact that  for a non-zero 
$\alpha \in \C^* $ we have $\sigma(\alpha z; \alpha \Lambda ) = \alpha\sigma(z;\Lambda )$. 
Now substituting the value of $\Omega_{\mb{v}}\left(\mb{\alpha^{-1}z};\Lambda_{\tau} \right)$ 
from Proposition \ref{Omega_in_theta} yields
\begin{equation}\label{main_rank_n} 
\Omega_{\mb{v}}(\mb{z};\Lambda)= 
           \left(\frac{2\pi i}{\alpha}\right)^{\displaystyle \sum_{j=1}^{n} v_j^2-
\hspace{-0.7em}\sum_{1\leq j <k \leq n}\hspace{-0.7em} v_jv_k-1}\prod_{j=1}^{n} u_j^{
\mathlarger{ \frac{v_j^2-v_j}{2}}}\frac{\displaystyle 
\mathlarger{\theta}\Big(\prod_{j=1}^{n} u_j^{v_j},q \Big)}{\displaystyle \prod_{j=1}^{n}
\theta(u_j, q)^{2v_j^2-\sum_{k=1}^n v_jv_k } \hspace{-0.7em}\prod_{1\leq j <k \leq n}\hspace{-0.7em}\theta(u_ju_k, q)^{v_jv_k} },
\end{equation}
where
\begin{align*}
\theta(u_j,q) &= (1-u_j)\prod_{m\geq 1} \frac{(1-q^mu_j)(1-q^mu_j^{-1})}{(1-q^m)^2}      
                                                                        \label{theta_u_i}.
                                                                        \end{align*}

 In the following two sections, we compute the parity of terms in the right hand side of \eqref{main_rank_n}. 

%

\section{Proof of Theorem \ref{main_theorem_2}}

\begin{proposition} 
\label{three-one}
Assume the assumptions of Theorem \ref{main_theorem_1} and let $\mathlarger{\theta}\Big(\prod_{j=1}^{n} u_j^{v_j},q \Big)$ be as defined in \eqref{main_rank_n}. Then if there exists a non-negative integer $k$ such that $u_1, u_2, \ldots, u_k<0$ and $u_{k+1}, u_{k+2}, \ldots, u_n >0$,  we have
\begin{align*}
\Parity\left[\mathlarger{\theta}\Big(\prod_{j=1}^{n} u_j^{v_j},q \Big)\right] & \equiv \begin{cases}
\left\lfloor ~\sum_{i=1}^{n}v_i \beta_i ~\right\rfloor \pmod 2& {\rm if}
                                        ~\sum_{i=1}^{k}v_i~ {\rm is~ even},\\ 
                                        0 \pmod 2&{\rm if}~\sum_{i-1}^k v_i~{\rm is~ odd,}       \end{cases}
                                        \end{align*}
where $\beta_i$ is given in Table \ref{table:silbeta}.
\end{proposition}

\begin{proof}
Let $u_1,u_2,u_3,\ldots,u_k < 0$ and $u_{k+1},u_{k+2},u_{k+3},\ldots,u_n > 0$. (Note that for $k=0$, this reduces to $u_i >0$ for $1\leq i \leq n$.)  For all 
$u_i<0$ we can write $u_i = (-1)|u_i| .$ Thus the expansion for 
$\mathlarger{\theta}\Big(\prod_{i=1}^{n} u_i^{v_i},q \Big)$ can be  rewritten as
\begin{equation}
\Big(1- (-1)^{\sum_{i=1}^{k}v_i}\prod_{i=1}^{n}|u_i|^{v_i}\Big) \\
   \prod_{m\geq 1}\frac{(1-q^m(-1)^{\sum_{i=1}^{k}v_i}\prod_{i=1}^{n}|u_i|^{v_i})(1-q^m
           (-1)^{\sum_{i=1}^{k}v_i}\prod_{i=1}^{n}|u_i|^{-v_i})}{(1-q^m)^2}.
           \label{prodtheta_ui}
           \end{equation}
We consider cases according to the sign of $q$.

{\bf Case I.} Suppose that $q>0$. Then from the above expression we deduce that if $\sum_{i=1}^{k}v_i$ is odd then 
$\theta\Big(\prod_{i=1}^{n} u_i^{v_i},q \Big)$ is positive. 
For the case that 
$\sum_{i=1}^{k}v_i$ is even, the factor $1-\prod_{i=1}^{n}|u_i|^{v_i}$ may be positive or negative. 
Thus we further split into two cases. \\ \\
\emph{Subcase I.} Assume that $1-\prod_{i=1}^{n}|u_i|^{v_i} > 0.$\\
We observe that for all $m\geq 1$ we have $q^m<1,$ and so $1-q^m\prod_{i=1}^{n}|u_i|^{v_i} >0.$ However,
\begin{equation*}
1-q^m\prod_{i=1}^{n}|u_i|^{-v_i} <0 \quad 
\iff  m < \sum_{i=1}^{n}v_i\log_q|u_i|.
\end{equation*}
Hence for this case there are $\big\lfloor ~\sum_{i=1}^{n}v_i\log_q|u_i| ~\big\rfloor $ negative
signs in the expression 
\eqref{prodtheta_ui} for  $\mathlarger{\theta}\Big(\prod_{i=1}^{n} u_i^{v_i},q 
\Big).$\\\\ 
\emph{Subcase II.} Assume that $1-\prod_{i=1}^{n}|u_i|^{v_i} < 0.$\\
Following a similar argument used in the Subcase I we have that, 
\begin{equation*}
1-q^m\prod_{i=1}^{n}|u_i|^{v_i} <0 \quad \iff \quad m < \sum_{i=1}^{n}-v_i\log_q|u_i|.
\end{equation*}
Observe that since $1-\prod_{i=1}^{n}|u_i|^{v_i}<0,$ we have $\sum_{i=1}^{n}-v_i\log_q|u_i| > 0.$ 
Hence there are in total $\big\lfloor -\sum_{i=1}^{n}v_i\log_q|u_i| ~ \big\rfloor +1 $ negative signs in 
expression \eqref{prodtheta_ui} for  $\mathlarger{\theta}\Big(\prod_{i=1}^{n} u_i^{v_i},q \Big).$ 
(The addition of 1 in the count of negative signs comes from the factor $1-\prod_{i=1}^{n}|u_i|^{v_i} .$)

Now we claim that the number $ \sum_{i=1}^{n}v_i\log_q|u_i| $ is not an integer. More generally, we 
claim that $\log_q|u_1|,\log_q|u_2|,\ldots,\log_q|u_n|, $ and 1 are linearly independent over $\Q.$ To 
see this suppose that there are integers $k_0, k_1,k_2,\ldots,k_n $ not all zero such that 
the sum $ \sum_{i=1}^{n}k_i\log_q|u_i| +k_0 =  0.$ Equivalently we have that $ \sum_{i=1}^{n}
k_iz_i = -k_0.$ Note that we have $1\in \Lambda_{\tau}$, 
hence under the isomorphism $\C/\Lambda_{\tau} \cong E(\C)$ integers are mapped to the 
identity element of $E(\C).$ Thus $ \sum_{i=1}^{n}k_iz_i = -k_0$ under the isomorphism 
$\C/\Lambda_{\tau} \cong E(\C)$ leads to $ \sum_{i=1}^{n}k_iP_i =\mathcal{O}.$  This 
contradicts our assumption that the points $P_1,P_2,\ldots,P_n$ are linearly independent in $E(\R).$ 
Hence we have that $\log_q|u_1|,\log_q|u_2|,\ldots,\log_q|u_n|, $ and 1 are linearly independent over 
$\Q.$ (This also shows that each number $\log_q|u_i|$ is irrational.) Therefore the number 
$ \sum_{i=1}^{n}v_i\log_q|u_i|$ can not be an integer. Using this fact and the property of the greatest integer function that
\begin{equation}\label{floor}
\lfloor x \rfloor + \lfloor -x \rfloor = 
 \begin{cases} 0 &\mbox{if}~ x\in \Z, \\
              -1 &\mbox{if}~ x\not\in \Z,
 \end{cases}
\end{equation}
we see that the number of negative signs in Subcase II is $-\big\lfloor \sum_{i=1}^{n}v_i\log_q|u_i| ~ \big\rfloor $.
Therefore we can combine 
the results from these two sub-cases to get that
\begin{align}\label{parqp3}
\Parity\left[\mathlarger{\theta}\Big(\prod_{i=1}^{n} u_i^{v_i},q \Big)\right] 
       &\equiv \left\lfloor ~\sum_{i=1}^{n}v_i\beta_i ~\right\rfloor \pmod 2 \quad \text{if}
                                        ~\sum_{i=1}^{k}v_i~ \text{is even},                                       
\end{align}
where $\beta_i =\log_q|u_i|$ for all $1\leq i \leq n.$\\ 

{\bf Case II.}  Suppose that $q <0$.
Let $x=\prod_{i=1}^n u_i^{v_i}$. Note that in this case $u_i>0$ for $1\leq i \leq n$,
hence $x>0$.
From definition of $\theta$ we have
\begin{eqnarray*}
	\theta(\prod_{i=1}^n u_i^{v_i}, q) &=& \theta(x,q) \\
	&=& (1-x) \prod_{m\geq 1} {(1-xq^m)(1-xq^{-m}) \over (1-q^m)^2} \\
	&=& (1-x) \left(\prod_{m\geq 1} {(1-xq^{2m})(1-xq^{-2m}) \over (1-q^{2m})^2} \right)
		\left(\prod_{m\geq 1} {(1-xq^{2m+1})(1-xq^{-2m-1}) \over (1-q^{2m+1})^2} \right) \\
	&=&  \theta(x, q^2)
		\prod_{m\geq 1} {(1-xq^{2m+1})(1-xq^{-2m-1}) \over (1-q^{2m+1})^2}
\end{eqnarray*}
Note that $1-xq^{2m+1}$ and $1-xq^{-2m-1}$ are both positive, since $q$ is assumed to be
negative. As a result
\[ \prod_{m\geq 1} {(1-xq^{2m+1})(1-xq^{-2m-1}) \over (1-q^{2m+1})^2} > 0,\]
and we get $\Sign\left[\theta(x,q)\right] = \Sign\left[\theta(x,q^2)\right].$
Since $q^2>0$ and $u_i>0$, applying  \eqref{parqp3} we get
\begin{align}
\Parity\left[\mathlarger{\theta}\Big(\prod_{i=1}^{n} u_i^{v_i},q \Big)\right] 
       \equiv
\Parity\left[\mathlarger{\theta}\Big(\prod_{i=1}^{n} u_i^{v_i},q^2 \Big)\right] 
       \equiv
			 \left\lfloor ~\sum_{i=1}^{n}v_i\beta_i ~\right\rfloor \pmod 2
\end{align}
where $\beta_i = \log_{q^2} u_i = {1\over 2}\log_{|q|} u_i$.
\end{proof}

We record two immediate corollaries from this proposition, which we will use
in next section.
\begin{corollary}
	\label{corui}
	Assume that $u_i$ and $q$ are normalized so that if $q>0$ then 
	$q<|u_i|<1$, and for $q<0$ we have $q^2<u_i<1.$
	Then $\theta(u_i, q)>0$.
\end{corollary}
\begin{proof}
If $q>0$ and $u_i<0$, then by Proposition \ref{three-one} for $v_i=1$ (odd) we have
\[\Parity\left[\theta(u_i,q)\right]\equiv  0 \pmod 2. \]
Also if $q>0$ and $u_i>0$ or $q<0$, then by Proposition \ref{three-one} for $k=0$ (even), we have
	\[\Parity\left[\theta(u_i,q)\right]\equiv \lfloor \beta_i \rfloor = 0 \pmod 2, \]
	since $0<\beta_i<1$.
	Thus in both cases $\theta(u_i,q)$ is positive.
\end{proof}

\begin{corollary}
	\label{coruij}
	Assume that $u_i$, $q$, and $\beta_i$ are defined as in Proposition \ref{three-one}.
	Then 
	\[\Parity\left[\theta(u_iu_j, q)\right] \equiv \begin{cases}
			\lfloor \beta_i + \beta_j \rfloor \pmod 2 & \mbox{ if $u_i u_j>0,$} \\
			0 \pmod 2 & \mbox{ if $u_i u_j<0,$} 
		\end{cases}
	\]
\end{corollary}
\begin{proof}
It follows from the result of Proposition \ref{three-one}.
\end{proof}

We now proceed with the main proof of this section.
\begin{proof}[Proof of Theorem \ref{main_theorem_2}]
First of all note that for a non-singular non-degenerate elliptic
net $W: \Z^n \longrightarrow \R $ there exists an elliptic curve $E$ defined over $\R$ and a 
collection $\mb{P} = (P_1,~P_2,\ldots,P_n )$ of points in $E(\R),$ such that 
\begin{equation*}
W(\mb{v}) = f(\mb{v})\Psi_\mb{v}(\mb{P};E)
\end{equation*}
for any $\mb{v} \in \Z^n.$ Here $f:\Z^n \longrightarrow \R^*$ is a quadratic form and 
$\Psi (\mb{P};E)$ is the elliptic net associated to $\mb{P}$ and $E.$
Moreover, since $W(\mb{v})\neq 0$ for $\mb{v}\neq \mb{0}$ we
have that $P_1, P_2, \ldots, P_n$ are $n$ linearly independent points in $E(\mathbb{R})$ (See \cite[Theorem 7.4]{Stange}).
Next observe that in the expression \eqref{main_rank_n} the numbers $u_j$ and $q$ are in $\R^*.$ Therefore the product containing
$u_j$ and $q$ are also in $\R.$ 
Also by \cite[Theorem 4.4]{Stange}, since $E$ is defined over $\mathbb{R}$ then $\Psi_\mb{v}(\mb{P};E)\in \mathbb{R}$.
Hence from \eqref{main_rank_n} 
we conclude that $(2 \pi i/\alpha)^{\mathlarger{\sum_{i=1}^{n} v_i^2-\sum_{1\leq i <j 
\leq n} v_iv_j-1}} \in \R^*.$ Note that this statement is true for all $\mb{v} \in \Z^n,$ 
therefore for $n\geq 2$, taking $v_1 = 1, v_2 = 2$ and $v_i = 0$ for all $3 \leq i \leq n,$ we get that $(2\pi i/
\alpha)^2 \in \R^*.$ Furthermore, taking $v_1 = 2$ and $v_i = 0$ for all $2 \leq i \leq n,$ shows 
that $(2 \pi i/\alpha)^3 \in \R^*. $ Since $(2 \pi i/\alpha)^2$ and $(2 \pi i/\alpha)^3 \in \R^*$, 
we have that $2\pi i/\alpha \in \R^*. $  A similar result also holds if $n=1$, by choosing $v_1=2$ and $3$. Hence $2\pi i/\alpha$ is either a positive real number or a 
negative real number. Thus, after possibly replacing $W(\mb{v})$ with 
{\tiny $(-1)^{ \displaystyle {\sum_{i=1}^{n} v_i^2-\hspace{-0.7em} \sum_{1\leq i <j \leq n} 
\hspace{-0.7em}v_iv_j-1}}$}\hspace{-0.7em} $W(\mb{v})$, we have

\begin{align}\label{S(W)}
\Sign[W(\mb{v})] &= \Sign\left[g(\mb{v})\right]~\Sign\left[\prod_{i=1}^{n}u_i^{(v_i^2-v_i)/2}\right]~\Sign \left[ \mathlarger{\theta}\Big(\prod_{j=1}^{n} u_j^{v_j},q \Big)\right], 
\end{align}
where $$g(\mb{v})=  \frac{f_1(\mb{v})}{\displaystyle \prod_{j=1}^{n}
\theta(u_j, q)^{2v_j^2-\sum_{k=1}^n v_jv_k } \hspace{-0.7em}\prod_{1\leq j <k \leq n}\hspace{-0.7em}\theta(u_ju_k, q)^{v_jv_k} }.    $$
Here $f_1(\mb{v})=${\tiny $(-1)^{ \displaystyle {\sum_{i=1}^{n} v_i^2-\hspace{-0.7em} \sum_{1\leq i <j \leq n} 
\hspace{-0.7em}v_iv_j}}$}\hspace{-0.7em} $f(\mb{v})$, if $W(\mb{v})$ was replaced by
{\tiny $(-1)^{ \displaystyle {\sum_{i=1}^{n} v_i^2-\hspace{-0.7em} \sum_{1\leq i <j \leq n} 
\hspace{-0.7em}v_iv_j-1}}$}\hspace{-0.7em} $W(\mb{v})$, otherwise $f_1(\mb{v})=f(\mb{v})$.
Observe that $g(\mb{v})$ is a quadratic form. From \eqref{S(W)} we have
\begin{align}\label{S(W)1}
\Parity[W(\mb{v})] &= \Parity\left[g(\mb{v})\right]+\Parity\left[\prod_{i=1}^{n}u_i^{(v_i^2-v_i)/2}\right]+\Parity\left[ \mathlarger{\theta}\Big(\prod_{j=1}^{n} u_j^{v_j},q \Big)\right]. 
\end{align}
We next deal with $\Parity\left[\prod_{i=1}^{n}u_i^{(v_i^2-v_i)/2}\right]$. If all $u_i>0$ this value is zero. Now assume that $u_1,u_2,u_3,\dots, u_k < 0$ and  $u_{k+1},u_{k+2},u_{k+3}\dots u_n > 0.$ Looking at values of $v_i$ modulo $4$, we get that
\begin{align}\label{parqp_4}
\Parity\left[\prod_{i=1}^{n} u_i^{(v_i^2-v_i)/2}\right] 
&\equiv \sum_{i=1}^{k} \left\lfloor \frac{v_i}{2} \right\rfloor \pmod 2. 
\end{align}
Next  we define $H:\Z^n \longrightarrow \Z $ as follows.
If $u_1,u_2,u_3,\dots, u_k < 0$ and  $u_{k+1},u_{k+2},u_{k+3}\dots u_n > 0, $ we set 

\begin{align*}                        
H(\mb{v}) &=
\begin{cases}
\displaystyle  \Big\lfloor \sum_{i=1}^{n}v_i\beta_i \Big\rfloor + \sum_{i=1}^{k} 
                   \Big\lfloor \frac{v_i}{2} \Big\rfloor 
                 & \text{if $ \displaystyle ~\sum_{i=1}^{k}v_i$ is even, } \\ 
\displaystyle ~ \sum_{i=1}^{k} \Big\lfloor \frac{v_i}{2}\Big \rfloor  & \text{if $\displaystyle ~\sum_{i=1}^{k}v_i$ is odd. }  
\end{cases}
\end{align*}
From \eqref{S(W)1}, Proposition \ref{three-one}, \eqref{parqp_4}, and the expressions for $H(\mb{v})$, we conclude that 

\begin{equation*}
\Parity[g(\mb{v})W(\mb{v})] \equiv H(\mb{v}) \pmod 2.
\end{equation*}
The proof is complete.
\end{proof}

\section{Proof of Theorem \ref{main_theorem_1}}
\begin{proof}[Proof of Theorem \ref{main_theorem_1}]
First of all note that in the proof of Proposition \ref{three-one} we showed that $\beta_1, \ldots, \beta_n$ are $n$ irrational numbers that are linearly independent over $\mathbb{Q}$. Moreover, 
as described in the proof of Theorem \ref{main_theorem_2}, $2\pi i/\alpha$ in \eqref{main_rank_n} is a non-zero real number. From now on, without loss of generality, we will assume that 
$2\pi i/\alpha>0.~  ($Note that if $2\pi i/\alpha < 0$ we can compute the sign of $~\Omega_{\mb{v}}
(\mb{z}; \Lambda)$ by considering $(-1)^{ \mathlarger {\sum_{i=1}^{n} v_i^2- \sum_{1\leq i <j \leq 
n} v_iv_j-1}}\Omega_{\mb{v}}(\mb{z};\Lambda).$) Since $2 \pi i\alpha^{-1} >0, $ it does not play any 
role in determining the sign of \eqref{main_rank_n}. Thus from \eqref{main_rank_n} we have that the 
Parity$\big[\Omega_{\mb{v}}(\mb{z};\Lambda)\big]$ in $\Z/2\Z $ is equal to 
\begin{multline}\label{main_parity}
       \Parity\left[\prod_{i=1}^{n} u_i^{(v_i^2-v_i)/2}\right]  
    +  \Parity\left[\mathlarger{\theta}\Big(\prod_{i=1}^{n} u_i^{v_i},q \Big)\right] \\  
    +  \Parity\left[\prod_{i=1}^{n}\theta(u_i, q)^{2v_i^2-\sum_{j=1}^n v_iv_j }\right] 
    +  \Parity\left[\prod_{1\leq i <j \leq n}\theta(u_iu_j, q)^{v_iv_j} \right].
\end{multline}
The first two terms of the above sum were computed in \eqref{parqp_4} and Proposition \ref{three-one} respectively.
By Corollary \ref{corui}, we get that $\theta(u_i,q)>0$, so the third summand is even. Thus, 
\begin{equation}
\label{last}
 \Parity\left[\prod_{i=1}^{n}\theta(u_i, q)^{2v_i^2-\sum_{j=1}^n v_iv_j }\right] \equiv 0 \pmod{2}.
 \end{equation}
 
Finally, for the last summand we have
\begin{eqnarray*}
	\Parity\left[\prod_{1\leq i <j \leq n}\theta(u_iu_j, q)^{v_iv_j} \right] & \equiv & 
	\sum_{1 \leq i<j \leq n} v_iv_j \Parity\left[\theta(u_iu_j,q)\right] \pmod 2.
\end{eqnarray*}
Note that in the range $1 \leq i<j \leq n$, we have $u_iu_j<0$ only when $1 \leq i \leq k < j \leq n$.
(That is, $u_iu_j>0$ when $1\leq i<j \leq k$ or $k+1 \leq i < j \leq n$.)
By Corollary \ref{coruij} we have
\[
	\Parity\left[\theta(u_iu_j,q)\right] \equiv \begin{cases}
		0 \pmod 2 & \mbox{ if $1 \leq i \leq k < j \leq n,$} \\
		\lfloor \beta_i + \beta_j \rfloor \pmod 2 & \mbox{ otherwise.}
	\end{cases}
\]
Therefore we get
\begin{align}
\label{final}
	\Parity\left[\prod_{1\leq i <j \leq n}\theta(u_iu_j, q)^{v_iv_j} \right] & \equiv \sum_{1 \leq i<j \leq n} v_iv_j \Parity\left[\theta(u_iu_j,q)\right] \nonumber \\ 
	&\equiv \sum_{1 \leq i < j \leq k} v_i v_j \lfloor{\beta_i + \beta_j} \rfloor
	 + \sum_{k+1 \leq i < j \leq n} v_i v_j \lfloor{\beta_i + \beta_j} \rfloor \pmod 2.
\end{align}

Now applying \eqref{parqp_4}, Proposition \ref{three-one}, \eqref{last}, and \eqref{final} in \eqref{main_parity} yield
\begin{equation*}
     \Parity[\Omega_{\mb{v}}(\mb{z};\Lambda)]\equiv
\begin{cases}
\displaystyle \sum_{1\leq i<j\leq k}\lfloor \beta_i + \beta_j \rfloor v_iv_j + 
      \hspace{-0.9em} \sum_{k+1\leq i<j\leq n }\lfloor \beta_i + \beta_j \rfloor v_iv_j \\ 
+\displaystyle \Big\lfloor ~\sum_{i=1}^{n}v_i\beta_i ~\Big\rfloor + \sum_{i=1}^{k} \left\lfloor \frac{v_i}{2} \right\rfloor \pmod 2 \hspace{1.5cm} \text{if}~ \sum_{i=1}^{k}v_i~ \text{is even,}\\\\
\displaystyle \sum_{1\leq i<j\leq k}\lfloor \beta_i + \beta_j \rfloor v_iv_j + 
      \hspace{-0.9em} \sum_{k+1\leq i<j\leq n }\lfloor \beta_i + \beta_j \rfloor v_iv_j \\
 + \displaystyle \sum_{i=1}^{k} \left\lfloor \frac{v_i}{2} \right\rfloor \pmod 2 \hspace{3.8cm} \text{if}~ \sum_{i=1}^{k}v_i~ \text{is odd.}
\end{cases}
\end{equation*} 
\end{proof}

\section{Numerical Examples}
We now give illustrations of various cases of Theorem \ref{main_theorem_1} with the help of some
examples. For sake of simplicity we only give examples for rank 2 elliptic nets. 
All the computations were done using mathematical software  \texttt{SAGE.}

Keeping the assumptions and notations used in Theorem \ref{main_theorem_1}, for the case $n=2,$ 
the sign of either $\Psi_{\mb{v}}(\mb{P};E)$ or   $(-1)^{v_1^2+v_2^2-v_1v_2-1}\Psi_{\mb{v}}(\mb{P};E),$ can be computed using one of the following parity formulas:
\begin{align}
\Parity[\Psi_{\mb{v}}(P;E)] &\equiv \Big\lfloor v_1\beta_1 + v_2\beta_2 \Big\rfloor + 
     \Big\lfloor \beta_1 + \beta_2 \Big\rfloor v_1v_2 \pmod 2  \label{u1pu2p} \\
\Parity[\Psi_{\mb{v}}(P;E)] &\equiv
\begin{cases}
\displaystyle\Big\lfloor v_1\beta_1+v_2\beta_2 \Big\rfloor + \Big\lfloor \frac{v_1}{2} 
               \Big\rfloor \pmod 2 & \text{if $v_1$ is even, } \\   
\displaystyle \Big\lfloor \frac{v_1}{2} \Big\rfloor  \pmod 2 &  \text{if $v_1$ is odd. }
\end{cases} \label{u1nu2p} \\
\Parity[\Psi_{\mb{v}}(P;E)] &\equiv
\begin{cases}
\displaystyle \Big\lfloor v_1\beta_1+v_2\beta_2 \Big\rfloor + \Big\lfloor \frac{v_2}{2} 
      \Big\rfloor \pmod 2 & \text{if $v_2$ is even, } \\
\displaystyle  \Big\lfloor \frac{v_2}{2} \Big\rfloor  \pmod 2 &  \text{if $v_2$ is odd. }
\end{cases} \label{u1pu2n} \\
\Parity[\Psi_{\mb{v}}(P;E)] &\equiv
\begin{cases}
\Big\lfloor v_1\beta_1+v_2\beta_2 \Big\rfloor +\Big\lfloor \beta_1 + \beta_2 \Big\rfloor v_1v_2\\ 
   \hspace{6.5em}      \displaystyle + \Big\lfloor \frac{v_1}{2} \Big\rfloor + 
\Big\lfloor \frac{v_2}{2} \Big\rfloor   \pmod 2 & \text{if $v_1+v_2$ is even, } \\
\Big\lfloor \beta_1 + \beta_2 \Big\rfloor v_1v_2+ \displaystyle \Big\lfloor \frac{v_1}{2} 
\Big\rfloor + \Big\lfloor \frac{v_2}{2} \Big\rfloor   \pmod 2 &  \text{if $v_1+v_2$ is odd. }
\end{cases} \label{u1nu2n}
\end{align}
Here the two irrational numbers $\beta_1$ and $\beta_2 $ are given by the rules in Table 
\ref{table:silbeta}.
%
The formula \eqref{u1pu2p} is used when $u_1>0$ and $u_2>0$ and formula \eqref{u1nu2p} is used for the case when $u_1<0$ and $u_2>0.$ We use the formula \eqref{u1pu2n} when $u_1>0$ and $u_2<0.$ Finally the formula \eqref{u1nu2n} is used when both $u_1<0$ and $u_2<0.$

We have verified the truth of the above formulas for several rank 2 elliptic net $W(v_1,v_2)$ in the 
range  $0\leq v_1\leq 500$ and $0\leq v_2\leq 500 .$ Thus the results have been verified for 
$25\times 10^4$ of values of $W(v_1,v_2)$ and the same for the negative indices as well.
 
\begin{example}
Let $E$ be the elliptic curve defined over $\R $ given by the Weierstrass equation $ y^2 + x y = x^{3} -  x^{2} - 4 x + 4 .$ Let
$P_1 = (69/25, -532/125)$ and $P_2 = (2,-2)$ be two points in $E(\R).$ Let $\mb{P} = (P_1,P_2).$ 
The following table presents the values of $\Psi_{\mb{v}}(\mb{P};E)$ for $\mb{v} = 
(v_1,v_2)$ in the range $ 0 \leq v_1\leq 3$  and $ 0 \leq v_2\leq 5.$  
\begin{center} 
$\begin{array}{r r r r r r }
\hline
&            &             & \vdots          &              &                     \\
&-832 & 1232600000 & 430685595625000000 & 3330569636331576171875000000 &\\
&112 & -12560000 & -18772893750000 & 121093285553785156250000 &\\
\cdots &-4 & -165500 & -141878687500 & -1754232556789062500 & \cdots \\
&-2 & -150 & 196317500 & -1270400610718750  &\\
&1 & 95 & 152725 & -181061702375 &\\
& \mb{0} & 5 & -3595 & 63803440 &\\
&            &             & \vdots          &              &                    \\
\hline
\end{array}$
\captionof{table}{Elliptic net $\Psi(\mb{P};E)$ associated to elliptic curve $ E : y^2 + x y = x^{3} -  x^{2} - 4 x + 4 $ and points $P_1 = (69/25, -532/125),$ $P_2 = (2,-2).$}
\end{center}
%
%

In the above array the bottom left corner represents the value  $\Psi_{(0,0)}(\mb{P};E)$ and the 
upper right corner represents $\Psi_{(3,5)}(\mb{P};E).$

There is an isomorphism $E(\R)\cong \R^*/q^\Z $ such that $P_1\longleftrightarrow u_1$
and $P_2\longleftrightarrow u_2 $ with the explicit values
\begin{align*}
        q     &=  0.0001199632944492781512985480142643667840\ldots\ldots,   \\
        u_1   &=  0.0803285719586868777961922659399264909608\ldots\ldots, \\
        u_2   &=  0.03600942542966326797848808049477306988456\ldots\ldots       
\end{align*}
Since $u_1, u_2 > 0,$  by employing Theorem \ref{main_theorem_1}, the sign of 
$\Psi_{\mb{v}}(\mb{P};E)$ up to a factor of $(-1)^{v_1^2+v_2^2-v_1v_2-1}$ can be calculated by \eqref{u1pu2p}. Since Theorem \ref{main_theorem_1} gives either sign of $\Psi_{\mb{v}}(\mb{P};E)$ or 
$\displaystyle (-1)^{v_1^{2}+v_2^{2}-v_1v_2-1}\Psi_{\mb{v}}(\mb{P};E).$ 

By computing the sign of $\Psi_{(2,2)}(\mb{P};E)$ using \eqref{u1pu2p} 
we conclude that in this case the parity is given by the formula 
\begin{align*}
\Parity[\Psi_{\mb{v}}(\mb{P};E)] &\equiv \Big\lfloor v_1\beta_1 + v_2\beta_2 \Big\rfloor + 
                  \Big\lfloor \beta_1 + \beta_2 \Big\rfloor v_1v_2 + (v_1+v_2+v_1v_2+1) \pmod 2 
\end{align*}
with  \vspace{-1em}
\begin{align*}
 \beta_1 &= 0.2793020829801927957749331343976812416467 \ldots,  \\
     \beta_2 &= 0.3681717984734797193981452826601334954064 \ldots .
\end{align*}

Next we illustrate the truth of our formula using two special cases.
\begin{align*}
 {\rm Sign[}\Psi_{(1,3)}(\mb{P};E)] 
&= (-1)^{\lfloor \beta_1+3\beta_2\rfloor + 3 \lfloor \beta_1 + \beta_2 \rfloor + 8 } = -1  
 \end{align*}
and 
\begin{align*}
 \Sign[\Psi_{(3,4)}(\mb{P};E)] 
 &= (-1)^{\lfloor 3\beta_1+4\beta_2\rfloor + 12 \lfloor \beta_1 + \beta_2 \rfloor  + 20  } = 1,  
\end{align*}
which agree with the signs from the above table.
\end{example}

\begin{example}
Let $E$ be the elliptic curve defined over $\R $ given by the Weierstrass equation $ y^2 + x y = x^{3} -  x^{2} - 4 x + 4 .$ Let 
$P_1 = (-1, 3)$ and $P_2 = (3,2)$ be two points in $E(\R)$ so that $\mb{P} = (P_1,P_2) .$ The 
following table presents the values of $\Psi_{\mb{v}}(\mb{P};E)$ for $\mb{v} = (v_1,v_2)$ in the 
range $ 0 \leq v_1\leq 3$ and $ 0 \leq v_2\leq 6.$ 
\\ 
\begin{center}
$\begin{array}{r r r r r r}
\hline
& &  \vdots &  &  \\
& -219900856 & 71486913947 & 48178148140103 & -112925826309806338 &\\
& -495235 & 58762243 & 3246745150 & -20471103308793 &\\
& -749 & 170718 & -24093133 & -16532329817 &\\
\cdots & 62 & 2291 & -154139 & -28273396 &\cdots \\
& 7 & 67 & -1256 & -101083 &\\
& 1 & 4 & 3 & -1579 &\\
& \mb{0} & 1 & 5 & -94 &\\
& &  \vdots &  &  &\\
\hline
\end{array} $ 
\captionof{table}{Elliptic net $\Psi(\mb{P};E)$ associated to elliptic curve $ E : y^2 + x y = x^{3} -  x^{2} - 4 x + 4 $  and points $P_1 = (-1, 3),$ $P_2 = (3,-2).$}

\end{center}

In this case there is an isomorphism $E(\R) \cong \R^*/q^\Z $ such that $P_1\longleftrightarrow u_1 
$ and $P_2\longleftrightarrow u_2 $ with the explicit values
\begin{align*}
        q     &=  0.0001199632944492781512985480142643667840\ldots\ldots,   \\
        u_1   &= -0.283422955948679072053638499724508663516\ldots\ldots, \\
        u_2   &= 0.00129667871977447963166306014589504823338\ldots\ldots       
\end{align*}
Observe that $q $ is the same as in the previous example. 
Further since $u_1 < 0,$ and $u_2>0$, 
by using Theorem \ref{main_theorem_1}, parity of $\Psi_{\mb{v}}(\mb{P};E)$ is either given by 
\eqref{u1nu2p} or 
\begin{align}
\Parity[\Psi_{\mb{v}}(\mb{P};E)] &\equiv
\begin{cases}
\displaystyle\Big\lfloor v_1\beta_1+v_2\beta_2 \Big\rfloor + \Big\lfloor \frac{v_1}{2} 
               \Big\rfloor +v_2 + 1 \pmod 2 & \text{if $v_1$ is even} \\   
\displaystyle \Big\lfloor \frac{v_1}{2} \Big\rfloor  \pmod 2 &  \text{if $v_1$ is odd}
\end{cases}\label{parity_21}
\end{align}
 with 
\begin{align*}
 \beta_1 &= 0.1396510414900963978874665671988406208233 \ldots,  \\
     \beta_2 &= 0.7363435969469594387962905653202669908128 \ldots.
\end{align*}
 
By computing the sign of $\Psi_{(2,2)}(\mb{P};E)$ using \eqref{u1nu2p} and \eqref{parity_21} 
we conclude that in this case the parity is given by formula \eqref{parity_21}. Next we illustrate the truth of our formula using two special cases.
\begin{align*}
 {\rm Sign[}\Psi_{(2,3)}(\mb{P};E)] &= (-1)^{\Parity[\Psi_{(2,3)}(\mb{P};E)]} = (-1)^{\lfloor 2\beta_1+3\beta_2\rfloor + \lfloor 1 \rfloor + 3 + 1 }  = -1  
 \end{align*}
and 
\begin{align*}
 \Sign[\Psi_{(1,5)}(\mb{P};E)] 
  &= (-1)^{ \lfloor \frac{1}{2} \rfloor } = 1  
\end{align*}
Again these agree with the signs from the above table.
\end{example}
%

\begin{example}
Let $E$ be the elliptic curve defined over $\R $ given by the Weierstrass equation $ y^2 + y = x^{3} + x^{2} - 2 x .$ Let $P_1 = (-1, 1)$ and $P_2 = (0,-1)$ 
be two points in $E(\R).$ Let $\mb{P} = (P_1,P_2) .$ The following table presents the 
values of $\Psi_{\mb{v}}(\mb{P};E)$ for $\mb{v} = (v_1,v_2)$ in the range $ -5 \leq v_1\leq 5$  and 
$ -2 \leq v_2\leq 2.$\\  
\begin{center} 
$
\begin{array}{r r r r r r r r r r r r r }
\hline
& &            &      & &       & \vdots          &              &              &       \\
&535 & 44 & -7 & -1 & 1 & -1 & -4 & 17 & 151 & -55 & -106201 & \\
&1187 & 67 & 1 & -2 & -1 & 1 & 1 & -5 & 26 & 709 & -19061 &  \\
\cdots& -3376 & 129 & 19 & -3 & -1 & \mb{0} & 1 & 3 & -19 & -129 & 3376 & \cdots  \\
&19061 & -709 & -26 & 5 & -1 & -1 & 1 & 2 & -1 & -67 & -1187 & \\
&106201 & 55 & -151 & -17 & 4 & 1 & -1 & 1 & 7 & -44 & -535 &  \\
& &            &      & &       & \vdots          &              &              &       \\
\hline
\end{array}
$
\captionof{table}{Elliptic net $\Psi(\mb{P};E)$ associated to elliptic curve $ E :y^2 + y = x^{3} + x^{2} - 2 x  $ \\ and points $P_1 = (-1, 1),$ $P_2 = (0,-1).$}
\end{center}
The above array is centered at $\Psi_{(0,0)}(\mb{P};E) = 0.$ The bottom left corner represent the 
value  $\Psi_{(-5,-2)}(\mb{P};E)$ and the upper right corner represents $\Psi_{(5,2)}(\mb{P};E).$
For this example we have the isomorphism  $E(\R) \cong \R^*/q^\Z $ such that $P_1\longleftrightarrow u_1 $ 
and 
$P_2\longleftrightarrow u_2 $ with the explicit values
\begin{align*}
        q     &=  0.00035785976153723480818280896702856223292\ldots\ldots,   \\
        u_1   &=  -0.2170771835085414203450101536155224134341\ldots\ldots, \\
        u_2   &=  -0.0077622720300518161218942441500824493219\ldots\ldots.      
\end{align*}
Since $u_1< 0$ and $u_2<0$, by using Theorem \ref{main_theorem_1}, sign of $\Psi_{\mb{v}}
(\mb{P};E)$ is given by either 
\eqref{u1nu2n} or 
\begin{equation}\label{parity_31}
\Parity[\Psi_{\mb{v}}(P;E)] \equiv
\begin{cases}
\Big\lfloor v_1\beta_1+v_2\beta_2 \Big\rfloor +\Big\lfloor \beta_1 + \beta_2 \Big\rfloor v_1v_2 +
\\ \displaystyle \hspace{1em} \Big\lfloor \frac{v_1}{2} \Big\rfloor + \Big\lfloor \frac{v_2}{2} \Big\rfloor + 
v_1v_2+1    \hspace{1.5em}   \pmod 2 & \text{if $v_1+v_2$ is even }
\vspace{0.5em}  \\
\Big\lfloor \beta_1 + \beta_2 \Big\rfloor v_1v_2+ \displaystyle \Big\lfloor \frac{v_1}{2} 
\Big\rfloor + \Big\lfloor \frac{v_2}{2} \Big\rfloor   \pmod 2 &  \text{if $v_1+v_2$ is odd }
\end{cases}
\end{equation}
with 
\begin{align*}
 \beta_1 &= 0.1924929051139423228173765652973000996307 \ldots\ldots,  \\
     \beta_2 &= 0.6122563386959476420220464745591944344939\ldots\ldots.
\end{align*}
By computing the sign of $\Psi_{(2,2)}(\mb{P};E)$ using \eqref{u1nu2n} and \eqref{parity_31} 
we conclude that in this case the parity is given by formula \eqref{parity_31}.

\end{example}

\begin{example}
Let $E$ be the elliptic curve defined over $\R $ given by Weierstrass equation $ y^2 = x^3 - 7x + 10 .$ Let  
$P_1 = (-2,4) $ and $P_2 = (1,2)$ be two linear independent points in $E(\R).$ Let $\mb{P} = 
(P_1,P_2)$. The following table presents the values of $\Psi_{\mb{v}}(\mb{P};E)$ for $\mb{v} = 
(v_1,v_2)$ in range $ 0 \leq v_1\leq 4$ and $ 0 \leq v_2\leq 6.$ 
\begin{center}
$
\begin{array}{ r r r r r  r r }
\hline
&            &             & \vdots          &              &              &       \\
&-54525952 & 1086324736 & 81340137472 & -15800157077504 & -29481936481157120 &  \\
&-163840 & -950272 & 131956736 & 30954979328 & -31977195339776 & \\
 &-2048 & -17408 & 280576 & 85124096 & 30585993216 &   \\
\cdots& 32 & -352 & -9440 & 979488 & 449423648 & \cdots \\
 & 4 & -4 & -276 & -16028 & 8814788 & \\
& 1 & 3 & -31 & -1697 & 67225 & \\
& \mb{0} & 1 & 8 & -409 & -65488 &  \\
&            &             & \vdots          &              &              &       \\
\hline
\end{array}
$
\captionof{table}{Elliptic net $\Psi(\mb{P};E)$ associated to elliptic curve $ E :y^2 = x^3 - 7x + 10 $ \\ and points $P_1 = (-2, 4),$ $P_2 = (1,2).$}
\end{center}
In the above array the bottom left corner represents the value  $\Psi_{(0,0)}(\mb{P};E)$ and the 
upper right corner represents $\Psi_{(4,6)}(\mb{P};E).$
In this case there is an isomorphism $E(\R) \cong \R^*/q^\Z$ such that $P_1\longleftrightarrow u_1$ 
and $P_2\longleftrightarrow u_2 $ with explicit values
\begin{align*}
        q     &=  -0.0004077489822343239057667854741817549172\ldots\ldots,   \\
        u_1   &=  0.001201936348983837429349696735400418601519\ldots\ldots, \\
        u_2   &=  0.008992979917906651664620780969726498312814\ldots\ldots       
\end{align*}
Since $u_1, u_2>0$, by using Theorem \ref{main_theorem_1} and calculating the sign of 
$\Psi_{(2,2)}(\mb{P};E)$, we observe that the sign of $\Psi_{\mb{v}}(\mb{P};E)$ in this case is 
given by \eqref{u1pu2p} with 
\begin{align*}
 \beta_1 &= 0.4307458699792390794239197192204249668246 \ldots\ldots,  \\
     \beta_2 &= 0.3018191057841811111031361738974315389666\ldots\ldots.
\end{align*} 
\end{example}

\section{Uniform distribution of signs} 



\begin{definition}
Let $(S(\mb{v}))$ be an $n$-dimensional array of real numbers. For any $a$ and $b$ with
 $0\leq a< b \leq 1$ and for any positive integers $V_1,V_2,\ldots ,V_n$ denote 
\begin{equation*}
C\big([a,b); ~V_1,V_2,\ldots,V_n\big) = 
\#\Big\{  \mb{v}=(v_1, v_2, \ldots, v_n); 1  \leq  v_i \leq  V_i\hspace{.5em} \mbox{for}~  1 \leq i \leq n~and~
                 \{S(\mb{v})\} \in [a,b)  \Big\},
\end{equation*}
where $\{ S(\mb{v}) \}$ is the fractional part of $S(\mb{v}).$
Then the array $(S(\mb{v}))$ is said to be uniformly distributed mod 1 if 
\begin{equation*}
\lim_{V_1,V_2,\ldots ,V_n \to \infty} \frac{C([a,b);V_1,V_2,\ldots,V_n)}{V_1V_2\ldots V_n} = b-a.
\end{equation*}
\end{definition}

\begin{lemma}[\bf Weyl Criterion]\label{thm=WC}
The array $(S(\mb{v}))$ is uniformly distributed mod 1 if and only if 
\begin{equation*}
\lim_{V_1,V_2,\ldots ,V_n \rightarrow \infty} \frac{1}{V_1V_2\ldots V_n} 
       \sum_{1 \leq  v_1 \leq  V_1}\sum_{1  \leq  v_2 \leq  V_2}\ldots 
                    \sum_{1  \leq  v_n \leq  V_n} e^{2\pi i h S(\mb{v})} = 0
\end{equation*}
for all integers $h\neq 0.$
\end{lemma}
\begin{proof}
The proof follows along the same lines as the proof for $2$-dimensional case. See 
\cite[Chapter 1, Theorem 2.9]{UD}.
\end{proof}

\begin{proposition}\label{prop=udmodm}
Let $\theta_1$ be an irrational number and let $\theta_2,\theta_3,\ldots,\theta_n,$ and $\theta_0$ be arbitrary real
numbers. Then the array $(v_1\theta_1+v_2\theta_2+\cdots + v_n\theta_n+\theta_0)$ is uniformly 
distributed mod 1.
\end{proposition}
\begin{proof}
This is a direct consequence of Theorem \ref{thm=WC}. See also \cite[Example 2.9]{UD}.
\end{proof}


The following proposition is a generalization of a part of Theorem 3.1 of \cite{Niven} for  sequences to  arrays. 
\begin{proposition}\label{thm=ud_grank_modm}
For an irrational number $\theta_1$ and real numbers $\theta_2, \ldots \theta_n$, $\theta_0$, the array 
\begin{equation}\label{seq=form1}
 \Big(\Big\lfloor v_1\theta_1+v_2\theta_2+\cdots + v_n\theta_n+\theta_0\Big\rfloor \Big)
\end{equation}
is uniformly distributed mod $m.$
\end{proposition}
\begin{proof}
Proposition \ref{prop=udmodm} for the irrational number $\theta_1/m$ and real numbers 
$\theta_2/m, \ldots, \theta_n/m, \theta_0/m$ yields that the array of real numbers 
$(v_1\frac{\theta_1}{m}+v_2\frac{\theta_2}{m}+\cdots + v_n\frac{\theta_n}{m}+\frac{\theta_0}{m})$ is uniformly distributed mod $1.$ Thus we conclude 
that the array of fractional part 
\begin{equation*}
\left\{ v_1\frac{\theta_1}{m}+v_2\frac{\theta_2}{m}+\cdots + v_n\frac{\theta_n}{m} +\frac{\theta_0}{m}\right \} 
             = v_1\frac{\theta_1}{m}+v_2\frac{\theta_2}{m}+\cdots + v_n\frac{\theta_n}{m}+\frac{\theta_0}{m} 
    - \Big\lfloor v_1\frac{\theta_1}{m}+v_2\frac{\theta_2}{m}+\cdots + v_n\frac{\theta_n}{m}+\frac{\theta_0}{m}\Big\rfloor
\end{equation*}
is uniformly distributed in the unit interval $[0,1).$ By multiplying the terms of the array 
$(\{ v_1\theta_1/m+v_2\theta_2/m+\cdots + v_n\theta_n/m+\theta_0/m \})$ with $m$ we see that the array of 
real numbers
\begin{equation*}
   v_1\theta_1+v_2\theta_2+\cdots + v_n\theta_n+\theta_0 -
    m\Big\lfloor v_1\frac{\theta_1}{m}+v_2\frac{\theta_2}{m}+\cdots + v_n\frac{\theta_n}{m}+\frac{\theta_0}{m}\Big\rfloor
\end{equation*}
is uniformly distributed over the interval $[0,m)$ on the real line. Hence by taking the integer
parts of the terms of the above array we conclude that the terms of the array
\begin{multline}\label{eq=}
\Big\lfloor v_1\theta_1+v_2\theta_2+\cdots + v_n\theta_n+\theta_0 -m 
\Big\lfloor v_1\frac{\theta_1}{m}+v_2\frac{\theta_2}{m}+\cdots + v_n\frac{\theta_n}{m}+\frac{\theta_0}{m}\Big\rfloor\Big\rfloor \\
  =  \Big\lfloor v_1\theta_1+v_2\theta_2+\cdots + v_n\theta_n+\theta_0\Big\rfloor -
    m\Big\lfloor v_1\frac{\theta_1}{m}+v_2\frac{\theta_2}{m}+\cdots + v_n\frac{\theta_n}{m}+\frac{\theta_0}{m}\Big\rfloor
\end{multline}
are uniformly distributed modulo $m.$ 
Furthermore, the removal of the terms $m\lfloor v_1\theta_1/m+v_2\theta_2/m+\cdots 
+ v_n \theta_n/m+\theta_0/m \rfloor,$ form the array \eqref{eq=}, does not effect the uniform distribution mod 
$m.$ 
\end{proof}

\begin{corollary}\label{cor=ud}
Let $(S(\mb{v}))$ be an array of integers that is uniformly distributed mod $m.$ 
Let $(c(\mb{v}))$ be an integer array which is constant mod $m$. Then the array $(S(\mb{v})+c(\mb{v}))$ is uniformly 
distributed mod $m.$ In particular, under the assumptions of Theorem \ref{thm=ud_grank_modm}, the sequence 
\begin{equation}\label{seq=form2}
\Big(\Big\lfloor v_1\theta_1+v_2\theta_2+\cdots + v_n\theta_n +\theta_0\Big\rfloor +c(\mb{v})\Big)
\end{equation}
is uniformly distributed mod $m$ for a fixed real number $\theta_0$.
\end{corollary}
\begin{proof}
The first assertion follows from Definition \ref{def=ud_grank_modm}. The second one follows form Proposition \ref{thm=ud_grank_modm} and the first assertion.
\end{proof}

\begin{proof}[Proof of Theorem \ref{ud-sign}]
Let $(S(\mb{v}))$ be the $n$-dimensional array given by the formulas at the right-hand 
side of the congruences \eqref{eq=for_k=0},~\eqref{eq=for_k>0a}, and \eqref{eq=for_k>0b}. We show that $(S(\mb{v}))$ is uniformly distributed mod $2$. In order to do this, we consider $(S(\mb{v}))$ as union of $2^n$ subarrays   $(S_\ell(\mb{v}))$ ($1\leq \ell \leq 2^n$) according to the parity of $v_i$'s. It is enough to prove that $(S_\ell(\mb{v}))$ is uniformly distributed mod $2$.

For fixed $\ell$, $(S_\ell(\mb{v}))$ is one of the three formulas given in the right-hand 
side of the congruences \eqref{eq=for_k=0},~\eqref{eq=for_k>0a}, and \eqref{eq=for_k>0b}. We consider three cases.

{\bf Case I:}  From \eqref{eq=for_k=0} we have 
$$S_\ell(\mb{v})=  \left\lfloor \sum_{i=1}^{n}v_i\beta_i \right\rfloor + \sum_{1\leq i <j\leq n} 
              \lfloor \beta_i+\beta_j  \rfloor v_iv_j .$$
              
Since $\beta_i$'s are fixed irrational numbers and the parity of $v_i$'s are fixed,               
$(\sum_{1\leq i <j\leq n} 
              \lfloor \beta_i+\beta_j  \rfloor v_iv_j )$ is a fixed array $(c(\mb{v})) $ mod $2$ and thus,  by Corollary \ref{cor=ud}, $(S_\ell(\mb{v}))$ is uniformly distributed mod $2$.

{\bf Case II:}  Similar to Case I, $(\sum_{1\leq i <j\leq k} 
              \lfloor \beta_i+\beta_j  \rfloor v_iv_j +\sum_{k+1\leq i <j\leq n} 
              \lfloor \beta_i+\beta_j  \rfloor v_iv_j )$ is a fixed array $(c(\mb{v}))$ mod $2$. Thus, from \eqref{eq=for_k>0a}, we have
  
 \begin{eqnarray*}
 S_\ell(\mb{v})&\equiv& \displaystyle \Big\lfloor \sum_{i=1}^{n}v_i\beta_i \Big\rfloor 
                              + \sum_{i=1}^{k} \Big\lfloor \frac{v_i}{2} \Big\rfloor +c(\mb{v})
         \pmod 2\\
         &\equiv& \Big\lfloor  \sum_{i=1}^{k} \lfloor  \frac{v_i}{2}\rfloor (2\beta_i+1)+  \sum_{i=k+1}^{n} \lfloor  \frac{v_i}{2}\rfloor (2\beta_i)+\sum_{i=1}^{n}\eta_i\beta_i   \Big\rfloor+c(\mb{v}) \pmod 2,
 \end{eqnarray*}            
where $\eta_i\in \{0, 1\}$ according to the parity of $v_i$. Since $\beta_i$'s are fixed irrational numbers and the parity of $v_i$'s are fixed,               
$\sum_{i=1}^{n}\eta_i\beta_i$ is
a fixed real number $\theta_0$ and thus,  by Corollary \ref{cor=ud}, $(S_\ell(\mb{v}))$ is uniformly distributed mod $2$.

{\bf Case III:}  Similar to Case II, $(\sum_{1\leq i <j\leq k} 
              \lfloor \beta_i+\beta_j  \rfloor v_iv_j +\sum_{k+1\leq i <j\leq n} 
              \lfloor \beta_i+\beta_j  \rfloor v_iv_j )$ is a fixed array $(c(\mb{v}))$ mod $2$. Thus, from \eqref{eq=for_k>0b}, we have
\begin{eqnarray*}
 S_\ell(\mb{v})&\equiv&  \sum_{i=1}^{k} \Big\lfloor \frac{v_i}{2} \Big\rfloor +c(\mb{v})
         \pmod 2,
         \end{eqnarray*}
         which is uniformly distributed mod $2$ by Corollary \ref{cor=ud} and the fact that $\left({\sum_{i=1}^{k} \Big\lfloor \frac{v_i}{2} \Big\rfloor}    \right)$ is uniformly distributed mod $m$. This completes the proof.
\end{proof}
\begin{remark}
We remark that inclusion of the factor {\tiny $(-1)^{ \displaystyle {\sum_{i=1}^{n} v_i^2-\hspace{-0.7em} \sum_{1\leq i <j \leq n} 
\hspace{-0.7em}v_iv_j-1}}$}
does not affect the result of Theorem \ref{ud-sign}.
Note 
that
\begin{align*}
\Parity\left[{\tiny (-1)^{ \displaystyle {\sum_{i=1}^{n} v_i^2-\hspace{-0.7em} 
                            \sum_{1\leq i <j \leq n} \hspace{-0.7em}v_iv_j-1}}}\right] 
    &\equiv  { \displaystyle {\sum_{1\leq i \leq j \leq n} \hspace{-0.7em}v_iv_j+1}}  \pmod{2},    
\end{align*}
which is a constant for $v_i$'s with fixed parities. Thus, we can apply Corollary \ref{cor=ud}.
\end{remark}

\section{Relation with denominator sequences}

Let
\begin{equation}\label{Psi_hat}
\hat{\Psi}_{\mb{v}}(\mb{P};E) = F_{\mb{v}}(\mb{P})\Psi_{\mb{v}}(\mb{P};E) 
                                                 \qquad \text{for all $\mathbf{v}\in \Z^n$},
\end{equation}
where $F(\mb{P}): \Z^n \longrightarrow \Q^* $ is the quadratic form given by 
\begin{equation}\label{Q_F}
F_{\mb{v}}(\mb{P}) = \prod_{1\leq i\leq j \leq n }\gamma_{ij}^{v_iv_j},
\end{equation}
with 
\begin{equation*}
\gamma_{ii} = D_{\mb{e_i}\cdot \mb{P}} = D_{P_i}, \quad \text{and} ~~~ \gamma_{ij} = 
               \frac{D_{P_i+P_j}}{D_{P_i}D_{P_j}}  ~~\text{for}~ i\neq j.
\end{equation*}
We will need the following assertion proved in  \cite{ABY}.
\begin{proposition}\label{PrABY}
Let $E$ be an elliptic curve defined over $\Q$ given by the Weierstrass equation
\begin{equation*}
y^2 + a_1xy + a_3y = x^3+a_2x^2 + a_4x + a_6, \qquad a_i \in \Z. 
\end{equation*}
Let $\mathbf{P} =  (P_1,P_2,\ldots,P_n)$ be an $n$-tuple of linearly independent points in $E(\Q)$
so that each $P_i \pmod{\ell}$ is non-singular for every prime $\ell .$ Then we have 
 \begin{equation}\label{DP_psi_hat}
D_{\mb{v}\cdot P} = |\hat{\Psi}_{\mb{v}}(\mb{P};E)|
\end{equation}
for all $\mathbf{v} \in \Z^n .$
\begin{proof}
See \cite[Proposition 1.7]{ABY}.
\end{proof}
\end{proposition}

\begin{proof}[Proof of Theorem \ref{one-seven}]
First of all observe that since all the terms of a denominator net is positive hence the quadratic form given by \eqref{Q_F} is also positive. Therefore from \eqref{Psi_hat} it follows that    
\begin{equation}\label{par_psi}
 \Parity[\Psi_{\mb{v}}(\mb{P};E)] = \Parity[\hat{\Psi}_{\mb{v}}(\mb{P};E)].
\end{equation}
Now consider $W(\mb{v}):\Z^n\longrightarrow \Q $ such that 
\begin{equation*}
|W(\mb{v})| = D_{\mb{v}\cdot P}  \qquad \text{for all $\mathbf{v}\in \Z^n$},
\end{equation*}
and define
\begin{equation*}
\Sign[W(\mb{v})]   = (-1)^{\Parity[\Psi_{\mb{v}}(\mb{P};E)]},  
\end{equation*}
where the $\Parity[\Psi_{\mb{v}}(\mb{P};E)]$ is given in Theorem \ref{main_theorem_1}. 
Thus,
\begin{equation}\label{WVDP}
W(\mb{v}) = (-1)^{\Parity[\Psi_{\mb{v}}(\mb{P};E)]} D_{\mb{v}\cdot P}.
\end{equation}
By employing \eqref{DP_psi_hat} and \eqref{par_psi} we rewrite \eqref{WVDP} as 
\begin{align*}
W(\mb{v}) &= (-1)^{\Parity[\hat{\Psi}_{\mb{v}}(\mb{P};E)]}|\hat{\Psi}_{\mb{v}}(\mb{P};E)| \\
          &= \Sign[{\hat{\Psi}_{\mb{v}}(\mb{P};E)}]|\hat{\Psi}_{\mb{v}}(\mb{P};E)| \\
          &= \hat{\Psi}_{\mb{v}}(\mb{P};E). 
\end{align*}
 Hence $W(\mb{v})$ is an elliptic net by \cite[Proposition 6.1]{Stange} and the fact that 
 ${\Psi}_{\mb{v}}(\mb{P};E)$ is an elliptic net.
\end{proof}


\begin{thebibliography}{10} 
%
%
%

\bibitem{ABY}{A. Akbary, J. Bleaney, and S. Yazdani}, On symmetries of elliptic nets and 
                valuations of net polynomials, {\it Journal of Number Theory} {\bf 158} (2016), 1--23.



\bibitem{UD}{\sc L. Kuipers and H. Niederreiter}, Uniform Distribution of Sequences:
            {\it Dover Publications}, {2006.}
            
\bibitem{Lang1}{\sc S. Lang}, Elliptic Curves: Diophantine Analysis,
        {\it Springer-Verlag}{1978}.

\bibitem{Niven}{\sc I. Niven}, Uniform distribution of sequences of integers, 
         {\it Trans. Amer.Math. Soc.}, {\bf 98} {(1960)}, {52--61}.
 
\bibitem{Kumar}{\sc M. Kumar}, The Signs in an Elliptic Net, Master's Thesis, 
        {University of Lethbridge}, {2014}.
 
\bibitem{Shipsey}{\sc R. Shipsey}, {Elliptic {D}ivisibility {S}equences}, {Ph.D. Thesis, 
         \it Goldsmith's College, University of London}, {2000}.
                

\bibitem{Sil0}{\sc J. Silverman, J. Tate}, Rational Points on Elliptic Curves,
        {Undergraduate Text in Mathematics}, {\it Springer}, {1994}.
        
\bibitem{Sil2}{\sc J. Silverman}, The Arithmetic of Elliptic Curves, Second Edition,
        {Graduate Text in Mathematics 106}, {\it Springer, New York}, {2009}.

\bibitem{Sil3}{\sc J. Silverman}, Advanced Topic in the Arithmetic of Elliptic Curves,
        {Graduate Text in Mathematics 151}, {\it Springer-Verlag, New York}, {1994}.

\bibitem{Sil1}{\sc J. Silverman and N. Stephens}, The sign of an elliptic divisibility sequence, 
        {\it J. of Ramanujan Mathematical Society} {\bf 21} (2006), 1--17.




\bibitem{Stange}{\sc K. Stange}, Elliptic nets and elliptic curves, 
        {\it Algebra and Number Theory} {\bf 5} (2011), 197-229.
        
\bibitem{Swart}{\sc C. Swart}, {{Elliptic {C}urves and {R}elated {S}equences}}, 
        {PhD Thesis, Royal Holloway and Bedford New College, University of London}, {2003}.        
 
\bibitem{Ward}{\sc M. Ward}, Memoir on elliptic divisibility sequences,
        {\it Amer. J. Math.} {\bf 70} (1948), 31--74.
 




%

\end{thebibliography}
\end{document}